\documentclass[10pt,a4paper]{article}
\usepackage[utf8]{inputenc}
\usepackage[T1]{fontenc}
\usepackage[margin=3cm]{geometry}
\usepackage{charter}
\usepackage[english]{babel}

\usepackage{amsmath}
\usepackage{amsfonts}
\usepackage{amssymb}
\usepackage{mathtools}
\usepackage{blkarray}
\usepackage{multirow}
\usepackage[affil-sl]{authblk}
\usepackage{wasysym}
\usepackage{nicefrac}
\usepackage{graphicx}
\usepackage{xcolor}
\usepackage{dsfont}
\definecolor{darkred}{RGB}{150, 0, 0}
\definecolor{darkgreen}{RGB}{0, 150, 0}
\definecolor{darkblue}{RGB}{0, 0, 150}
\usepackage[breaklinks,colorlinks,citecolor=darkgreen,linkcolor=darkred,urlcolor=darkblue,bookmarks=false]{hyperref}
\usepackage{bookmark}
\usepackage[vlined,ruled,commentsnumbered,linesnumbered,norelsize]{algorithm2e}
\usepackage{amsthm}
\usepackage{thmtools}
\usepackage{thm-restate}
\usepackage[textsize=footnotesize]{todonotes}
\usepackage{enumerate}
\usepackage{etoolbox}
\usepackage{tikz}
\usepackage{booktabs}
\usepackage{placeins}
\usepackage{dsfont}
\usepackage[inline]{enumitem}
\usepackage{xspace}
\usepackage[font=footnotesize,labelfont=bf,skip=2pt]{caption}
\usepackage{subcaption}
\usepackage{calc}
\usepackage{siunitx}
\usepackage{blkarray}
\usepackage{bigstrut}
\usepackage{setspace}

\makeatletter
\newcounter{subsubparagraph}[subparagraph]
\renewcommand\thesubsubparagraph{%
  \thesubparagraph.\@arabic\c@subsubparagraph}
\newcommand\subsubparagraph{%
  \@startsection{subsubparagraph}    
    {6}                              
    {\parindent}                     
    {3.25ex \@plus 1ex \@minus .2ex} 
    {-1em}                           
    {\normalfont\normalsize\bfseries\em}}
\newcommand\l@subsubparagraph{\@dottedtocline{6}{10em}{5em}}
\newcommand{\subsubparagraphmark}[1]{}
\makeatother

\DeclareMathAlphabet{\mathrm}{OT1}{bch}{m}{n}

\DeclareSIUnit{\byte}{B}

\newcommand{\binary}{\{0, 1\}}

\newcommand{\Continue}{\textbf{continue}}

\renewcommand{\P}{\ensuremath{\mathrm{P}}}
\newcommand{\NP}{\ensuremath{\mathrm{NP}}}

\newcommand{\define}{\coloneqq}
\newcommand{\card}[1]{|#1|}

\newcommand{\R}{\mathds{R}}

\newcommand{\Z}{\mathds{Z}}

\newcommand{\teams}{\ensuremath{T}}
\newcommand{\rounds}{\ensuremath{R}}
\newcommand{\match}{\ensuremath{m}}
\newcommand{\matches}{\ensuremath{\mathcal{M}}}
\newcommand{\matching}{\ensuremath{M}}
\newcommand{\matchings}{\ensuremath{\mathfrak{M}}}
\newcommand{\perms}{\ensuremath{\Pi}}
\newcommand{\permswo}[1]{\ensuremath{\perms^{-#1}}}
\newcommand{\permswotor}[3]{\ensuremath{\perms^{\smash{-#1}}_{\smash{#2,#3}}}}

\newcommand{\lprelaxation}[1]{\ensuremath{v_{#1}^{\mathrm{LP}}}}
\newcommand{\ipsolution}{\ensuremath{v^{\mathrm{IP}}}}

\newcommand{\etal}{et~al.\@\xspace}

\theoremstyle{plain}
\newtheorem{theorem}{Theorem}[section]
\newtheorem{lemma}[theorem]{Lemma}
\newtheorem{proposition}[theorem]{Proposition}

\newtheorem*{claim*}{Claim}

\theoremstyle{definition}
\newtheorem{remark}[theorem]{Remark}
\newtheorem{definition}[theorem]{Definition}
\newtheorem{example}[theorem]{Example}

\newcounter{problemct}
\newtheorem{problem}[problemct]{Problem}

\let\OLDthebibliography\thebibliography
\renewcommand\thebibliography[1]{
  \OLDthebibliography{#1}
  \setlength{\parskip}{.2em}
  \setlength{\itemsep}{.2em plus 0.3ex}
}

\author{Jasper van Doornmalen}
\author{Christopher Hojny\footnote{Corresponding author}}
\author{Roel Lambers}
\author{Frits C.R. Spieksma}
\title{Integer Programming Models \\ for Round Robin Tournaments}
\affil{Eindhoven University of
Technology, Department of Mathematics and Computer Science, P.O. Box 513, 5600 MB Eindhoven, The Netherlands\\
\{m.j.v.doornmalen, c.hojny, r.lambers,
f.c.r.spieksma\}@tue.nl}
\date{}

\begin{document}
\maketitle


\begin{abstract}
Round robin tournaments are omnipresent in sport competitions and
beyond. We propose two new integer programming formulations for scheduling
a round robin tournament, one of which we call {\em the matching
  formulation}. We analytically compare their linear relaxations with the
linear relaxation of a well-known traditional formulation. We find that
the matching formulation is stronger than the other formulations, while its
LP relaxation is still being solvable in polynomial time. In addition, we
provide an exponentially sized class of valid inequalities for the matching
formulation. Complementing our theoretical assessment of the strength of
the different formulations, we also experimentally show that the matching
formulation is superior on a broad set of instances. Finally, we describe a
branch-and-price algorithm for finding round robin tournaments that is
based on the matching formulation.
\end{abstract}

\section{Introduction}

Integer programming continues to be a very popular way to obtain a
schedule for a round robin tournament.
The ability to straightforwardly model such a tournament, and next solve the resulting formulation using an integer programming solver, greatly facilitates
practitioners. Moreover, it is usually possible to add all kinds of specific local constraints
to the formulation that help addressing particular challenges.
We substantiate this claim of the widespread use of integer
programming by mentioning some of the works that use integer programming to arrive at a schedule for a round robin tournament.
Indeed, from the literature, it is clear that for national football leagues (which are predominantly organized according to a so-called double round robin format), integer programming-based techniques are used extensively to find schedules. Without claiming to be exhaustive we mention Alarc\'on et al.~\cite{AlarconEtAl2017}, Della Croce and Oliveri~\cite{DellaCroceOliveri2006}, Dur\'an et al~\cite{Duranetal2007,Duranetal2017,DuranEtAl2021}, Goossens and Spieksma~\cite{GoossensSpieksma2009}, Rasmussen~\cite{Rasmussen2008}, Recalde et al.~\cite{RecaldeEtAl2013}, Ribeiro and Urrutia~\cite{RibeiroUrrutia2012}. Other sport competitions that are organized in a round robin fashion (or a format close to a round robin) have also received ample attention: we mention Cocchi et al.~\cite{CocchiEtAl2018} and Raknes and Pettersen~\cite{RaknesPettersen2018} who use integer programming for scheduling volleyball leagues, Fleurent and Ferland~\cite{FleurentFerland1993} who use integer programming for scheduling a hockey league, Kim~\cite{Kim2019} and Bouzarth et~al.~\cite{bouzarth2021scheduling} for baseball leagues, Kostuk and Willoughby~\cite{KostukWilloughby2012} for Canadian football, Nemhauser and Trick~\cite{NemhauserTrick1998} and Westphal~\cite{Westphal2014} for basketball leagues. Further, there has been work on studying properties of the traditional formulation, among others,
by Trick~\cite{trick2002integer} and Briskorn and Drexl~\cite{bridre2009}. Well-known surveys are given by
Rasmussen and Trick~\cite{rastri2008}, Kendall et al.~\cite{kenetal2010} and Goossens and Spieksma~\cite{GoossensSpieksma2012}; we also refer to
Knust~\cite{knust}, who maintains an elaborate classification of literature on sports scheduling.
More recently, the international timetabling competition~\cite{ITC2021}
featured a round robin sports timetabling problem, and most of the
submissions for this competition used integer programming in some way to obtain a good
schedule.

All this shows that integer programming is one of the most preferred ways to find schedules for competitions organized via a round robin format.

In this paper, we aim to take a fresh look at the problem of finding an optimal schedule
for round robin tournaments using integer programming techniques. Depending upon how often a pair of teams is required to meet, different variations of a round robin tournament arise: in case each pair of teams meets once, the resulting format is called a Single Round Robin, in case each pair of teams is required to meet twice, we refer to the resulting variation as a Double Round Robin. These formats are the ones that occur most in practice; in general we speak of a $k$-Round Robin to describe the situation where each pair of teams is required to meet $k$ times.

We have organized the paper as follows. In Section~\ref{sec:models}, we precisely define the problem corresponding to the Single Round Robin tournament, and we present three integer programming formulations for it. We call them the traditional formulation (Section~\ref{sec:traditionalformulation}), the matching formulation (Section~\ref{sec:matchingformulation}), and the permutation formulation (Section~\ref{sec:permutationformulation}); the latter two formulations are, to the best of our knowledge, new. We show that their linear relaxations can be solved in polynomial time. We prove in Section~\ref{sec:strength} that the matching formulation is stronger than the other formulations. In Section~\ref{sec:inequalities} we provide a class of valid inequalities for the matching formulation. We show in Section~\ref{sec:kRR} how our results extend to the $k$-Round Robin tournament. In Section~\ref{sec:computationalresults}, we generate instances of our problem with two goals in mind: (i) to experimentally assess the quality of the bounds found by our models (Section~\ref{sec:computationalcomparison}), and (ii) to report on the performance of a branch-and-price algorithm (Section~\ref{sec:branchandprice}). We conclude in Section~\ref{sec:conclusion}.

\section{Problem definition and formulations}
\label{sec:models}

In this section, we provide a formal definition of our problem and introduce the necessary terminology and notation. We start by describing the so-called Single Round Robin (SRR) tournament, where every pair of teams has to meet exactly once, and we return to the general version of the problem, where every pair of teams has to meet $k$ times ($k \geq 1$), in Section~\ref{sec:kRR}.

Throughout the entire paper, we assume that~$n$ is an even integer that
denotes the number of teams; for reasons of convenience we assume $n \geq 4$. We denote the set of all teams by~$\teams$.
A \emph{match} is a set consisting of two distinct teams and the set of all
\emph{matches} is denoted by~$\matches$, in formulae, $\matches = \{
m=\{i,j\} : i,j \in \teams,\; i \neq j\}$. We denote, for each
$i \in \teams$, by $\matches_i = \{\{i,j\} : j \in \teams \setminus \{i\}\}$ the set of matches played by team $i$.
As we assume in this section that every pair of teams meets once, and as $n$ is even, the matches can be organized in~$n-1$ rounds, which we denote by~$\rounds$; hence, we deal in this section with a {\em compact} single round robin tournament.

Prepared with this terminology and notation, we are able to provide a
formal definition of the SRR problem.

\begin{problem} (SRR)
  \label{prob:srr}
  Given an even number~$n \geq 4$ of teams with corresponding matches~$\matches$,
  a set of~$n-1$ rounds~$\rounds$, as well as an integral cost~$c_{\match,r}$ for
  every match~$\match\in\matches$ and round~$r \in \rounds$, the \emph{single
    round robin (SRR)} problem is to find an assignment~$\mathcal{A}
  \subseteq \matches \times \rounds$ of matches to rounds that minimizes
  the cost~$\sum_{(\match,r) \in \mathcal{A}} c_{\match,r}$ such that every
  team plays a single match per round and each match is played in some round.
\end{problem}
Since the SRR problem is \NP-hard (see Easton~\cite{easton2002}, Briskorn \etal~\cite{briskorn2010round}, and Van Bulck and Goossens~\cite{bulgoo2020}), there does not exist a polynomial
time algorithm to find an optimal assignment unless~$\P = \NP$.
For this reason, several researchers have investigated integer programming (IP)
techniques for finding an optimal assignment of matches to rounds.
We follow this line of research and discuss three different IP formulations
for the SRR problem: a traditional formulation with polynomially many
variables and constraints (Section~\ref{sec:traditionalformulation}) as well as two formulations that involve
exponentially many variables (Sections~\ref{sec:matchingformulation} and \ref{sec:permutationformulation}).
To the best of our knowledge, the latter models have not been discussed in
the literature before.

\subsection{The traditional formulation}
\label{sec:traditionalformulation}

The \emph{traditional formulation} of the SRR problem has been discussed,
among others, by Trick~\cite{trick2002integer} and Briskorn and
Drexl~\cite{bridre2009}.
To model an assignment of matches to rounds, this formulation introduces,
for every match~$\match\in\matches$ and round~$r \in \rounds$, a binary
decision variable~$x_{\match,r}$ to model whether match~$\match$ is played
at round~$r$ ($x_{\match,r} = 1$) or not ($x_{\match,r} = 0$).
With these variables, problem SRR can be modeled as:
\begin{subequations}
  \makeatletter
  \def\@currentlabel{T}
  \makeatother
  \renewcommand{\theequation}{T\arabic{equation}}%
  \label{tra}
  \begin{align}
    \label{tra:obj}
    \min \sum_{\match\in\matches}\sum_{r \in \rounds} c_{\match,r}x_{\match,r} &&&\\
    \sum_{r \in R} x_{\match,r} &= 1, && \match\in\matches,\label{tra:matchplayed}\\
    \sum_{m \in \matches_i} x_{m,r} &= 1, && i \in \teams, r \in \rounds,\label{tra:teamplays}\\
    x_{\match,r} &\in \binary, && \match \in \matches, r \in \rounds.\label{tra:binary}
  \end{align}
\end{subequations}
Constraints~\eqref{tra:matchplayed} ensure that each pair of teams meets
once, and Constraints~\eqref{tra:teamplays} imply that each team
plays in each round.
This model has~$O(n^2)$ constraints and~$O(n^3)$ variables.
Note that Constraints~\eqref{tra:binary} can be replaced by~$x_{\match,r}
\in \Z_+$ as the upper bound~$x_{\match,r} \leq 1$ is implicitly imposed
via Constraints~\eqref{tra:matchplayed} and non-negativity of variables. The linear programming relaxation of \eqref{tra} arises when we replace~(\ref{tra:binary}) by $x_{m,r} \geq 0$; given an instance $I$ of SRR, we denote the resulting value by $v^{LP}_{tra}(I)$.

\subsection{The matching formulation}
\label{sec:matchingformulation}

Consider the complete graph that results when associating a node to each team, say
$K_n = (T, \matches)$. Clearly, a single round of a feasible schedule can be seen as a perfect matching in this graph. This observation allows us to build a matching based
formulation by introducing a binary variable for every perfect matching in $K_n$;
we denote the set of all perfect matchings in $K_n$ by~$\matchings$.

We employ a binary variable~$y_{\matching,r}$ for each perfect matching~$\matching\in\matchings$ and round~$r \in \rounds$.
If~\mbox{$y_{\matching,r} = 1$}, the model prescribes that matching~$\matching$ is used for the schedule of round~$r$, whereas \mbox{$y_{\matching,r} = 0$} encodes
that a different schedule is used.
To be able to represent the cost of round~$r \in R$, the total cost of all
matches in~$\matching$ is denoted by~$d_{\matching,r} \define
\sum_{\match\in\matching} c_{\match,r}$, which leads to the model
\begin{subequations}
  \makeatletter
  \def\@currentlabel{M}
  \makeatother
  \renewcommand{\theequation}{M\arabic{equation}}%
  \label{mat}
  \begin{align}
    \label{mat:obj}
    \min \sum_{\matching\in\matchings} \sum_{r \in \rounds}
    d_{\matching,r}y_{\matching,r} &&&\\
    \sum_{\matching \in \matchings} y_{\matching,r} &= 1, && r \in \rounds,\label{mat:matchingonround}\\
    \sum_{\substack{\matching \in \matchings\colon \\ \match \in \matching}}
    \sum_{r \in R} y_{\matching,r} &= 1, && \match \in \matches, \label{mat:matchplayed}\\
    y_{\matching,r} &\in \binary, && \matching \in \matchings, r \in \rounds.\label{mat:binary}
  \end{align}
\end{subequations}
Constraints~\eqref{mat:matchingonround} ensure that a matching is selected
in each round, while Constraints~\eqref{mat:matchplayed} enforce that each
pair of teams meets in some round.
Similarly to the traditional formulation, we can replace~\eqref{mat:binary}
by~$y_{\matching,r} \in \Z_+$.
In this way, the linear programming relaxation of \eqref{mat} arises when replacing~(\ref{mat:binary}) by $y_{\matching,r} \geq 0$; given an instance $I$ of SRR, the resulting value is denoted by $v^{LP}_{mat}(I)$. Notice that this formulation uses an exponential number of variables, as the number of matchings grows exponentially in $n$. Thus, a relevant question is whether we can find $v^{LP}_{mat}$ in polynomial time. The following observation shows that it can be answered affirmatively.

\begin{lemma}
  \label{lem:pricematch}
  The LP relaxation of the matching formulation~\eqref{mat} can be solved
  in polynomial time.
\end{lemma}
\begin{proof}
  Due to the celebrated result by Gr\"otschel et al.~\cite{GroetschelLovaszSchrijver1981}, it is sufficient to show
  that the separation problem for
  the constraints of the dual of the linear relaxation of Model~\eqref{mat} can be solved in polynomial time.
  To avoid an exponential number of variables in the dual, we replace
  Constraint~\eqref{mat:binary} by~$y_{\matching,r} \in \Z_+$ as explained
  above.
  Then, by introducing dual variables~$\alpha_r$, $r \in \rounds$, corresponding to Constraints~\eqref{mat:matchingonround} and~$\beta_\match$, $\match\in\matches$, corresponding to Constraints~\eqref{mat:matchplayed}, the constraints of the dual of the LP
  relaxation of Model~\eqref{mat} are:
  \begin{align*}
    \alpha_r + \sum_{\match \in \matching} \beta_{\match} & \leq d_{\matching,r},
    && \matching\in\matchings, r \in \rounds.
  \end{align*}
  Given values for the dual variables, say~$(\bar{\alpha}, \bar{\beta})$, the separation problem is to decide whether it satisfies all dual constraints.
  For fixed~$r \in \rounds$, we show that this problem can be solved in
  polynomial time.
  Thus, the assertion follows as there are only~$O(n)$ rounds.

  Indeed, if~$r \in \rounds$ is fixed, the problem reduces to check whether there
  exists a matching~$\matching\in\matchings$ such that
  \[
    \bar{\alpha}_r + \sum_{\match\in\matching} \bar{\beta}_\match
    >
    d_{\matching,r} = \sum_{\match\in\matching} c_{\match,r}
    \quad
    \Leftrightarrow
    \quad
    \sum_{\match \in \matching}(\bar{\beta}_\match - c_{\match,r}) > -\bar{\alpha}_r.
  \]
  The latter inequality asks whether there exists a perfect matching of teams with
  weight greater than~$-\bar{\alpha}_r$, where an edge~$\match$ between two
  teams is assigned weight~$(\bar{\beta}_\match - c_{\match,r})$.
  This problem can be solved in polynomial time by Edmonds' blossom
  algorithm~\cite{edmonds1965maximum,edmonds1965paths}, which concludes the
  proof.
\end{proof}

\subsection{The permutation formulation}

\label{sec:permutationformulation}
Instead of fixing the schedule of a round,
the \emph{permutation formulation} fixes, for a given team, the order of the teams against which the given team plays its successive matches.
That is, it introduces a variable for each team~$i$ and each permutation
of~$\teams \setminus \{ i \}$.
We denote the set of all such permutations by~$\permswo{i}$.
Moreover, for a team $j \in \teams$ and round $r \in \rounds$,
denote the set of all permutations where $j$ occurs at position $r$ in the
permutation by $\permswotor{i}{j}{r}$.
Permutations from $\permswotor{i}{j}{r}$ thus encode that team~$i$ plays
against team~$j$ on round $r$.
For a permutation $\pi \in \permswo{i}$ and round~$r \in R$, we refer to
the opponent of team~$i$ at round~$r$ as~$\pi_r \in \teams \setminus
\{i\}$.
The cost of a schedule encoded via permutations~$\permswo{i}$ for a team~$i
\in \teams$ is then given by~$e_{i,\pi} \define \sum_{r \in \rounds}
c_{\{i,\pi_r\}, r}$.
Using binary variables~$z_{i,\pi}$, where~$i \in \teams$ and~$\pi \in
\permswo{i}$, that encode whether~$i$ plays against its opponents in
order~$\pi$ ($z_{i,\pi} = 1$) or not ($z_{i,\pi}=0$), the permutation
formulation is
\begin{subequations}
  \makeatletter
  \def\@currentlabel{P}
  \makeatother
  \renewcommand{\theequation}{P\arabic{equation}}%
  \label{per}
  \begin{align}
    \min \frac{1}{2}\sum_{i \in \teams}\sum_{\pi \in \permswo{i}} e_{i,\pi}
    z_{i,\pi} &&& \label{per:obj}\\
    \sum_{\pi \in \permswo{i}} z_{i,\pi} &= 1, &&
    i \in \teams, \label{per:scheduleforteam}\\
    \sum_{\pi \in \permswotor{i}{j}{r}} z_{i,\pi} &= \sum_{\pi \in \permswotor{j}{i}{r}} z_{j,\pi},&&
    \{i, j\} \in \matches, r \in \rounds,\label{per:linking}\\
    z_{i,\pi} &\in \binary, && i \in \teams, \pi \in \permswo{i}.\label{per:binary}
  \end{align}
\end{subequations}
Constraints~\eqref{per:scheduleforteam}
ensure that a permutation is selected for each team,
while Constraints~\eqref{per:linking} enforce that,
given a round and a pair of teams, these teams meet in that round,
or they do not meet in that round.
Due to rescaling the objective by~$\frac{1}{2}$, we find the cost of an optimal SRR schedule.
Moreover, we can again replace Constraint~\eqref{per:binary} by~$z_{i,\pi}
\in \Z_+$.
The linear programming relaxation of \eqref{per} then arises when replacing Constraints~(\ref{per:binary}) by $z_{i,\pi} \geq 0$; given an instance $I$ of SRR, we denote the resulting value by $v^{LP}_{per}(I)$.

Since this model has~$n!$ variables, we again investigate whether its
LP relaxation can be solved efficiently.
\begin{lemma}
  \label{lem:pricepermutation}
  The LP relaxation of the permutation formulation~\eqref{per} can be
  solved in polynomial time.
\end{lemma}
\begin{proof}
  As in the proof of Lemma~\ref{lem:pricematch}, it is sufficient to show
  that the separation problem corresponding to the constraints of the dual of the relaxation of the permutation formulation can be
  solved in polynomial time.
  Again, to avoid exponentially many variables in the dual, we
  replace~\eqref{per:binary} by~$z_{i,\pi} \in \Z_+$.
  We introduce dual variables~$\alpha_i$ for each constraint of
  type~\eqref{per:scheduleforteam} and~$\beta_{\{i,j\},r}$ for each
  constraint of type~\eqref{per:linking}.
  To normalize Constraint~\eqref{per:linking}, we assume it to be given
  by~$\sum_{\pi \in \permswotor{i}{j}{r}} z_{i,\pi} - \sum_{\pi \in
    \permswotor{j}{i}{r}} z_{j,\pi} = 0$ with~$i < j$.
  Then, the dual constraints are given by
  \begin{align*}
    \alpha_i
    +
    \sum_{\substack{r \in \rounds\colon\\ i < \pi_r}} \beta_{\{i,\pi_r\},r}
    -
    \sum_{\substack{r \in \rounds\colon\\ i > \pi_r}} \beta_{\{i,\pi_r\},r}
    &\leq
    \frac{1}{2}e_{i,\pi}
    &&
    i \in \teams, \pi \in \permswo{i}.
  \end{align*}
  If~$i \in \teams$ is fixed, the separation problem for dual
  values~$(\bar{\alpha},\bar{\beta})$ is to decide whether there exists a
  permutation~$\pi \in \permswo{i}$ such that
  \[
    \sum_{\substack{r \in \rounds\colon\\ i < \pi_r}} \bar{\beta}_{\{i,\pi_r\},r}
    -
    \sum_{\substack{r \in \rounds\colon\\ i > \pi_r}} \bar{\beta}_{\{i,\pi_r\},r}
    -
    \frac{1}{2} \sum_{r \in \rounds} c_{\{i,\pi_r\},r}
    >
    -\bar{\alpha}_i
  \]
  due to the definition of~$e_{i,\pi}$.
  To answer this question, it is sufficient to find a permutation
  maximizing the left-hand side expression.
  Such a permutation can be found by computing a maximum weight perfect matching in
  the complete bipartite graph with node bipartition~$(\teams \setminus
  \{i\}) \cup \rounds$ and edge weights defined for each~$j \in \teams \setminus \{i\}$ and~$r \in \rounds$ by
  \[
    w_{j,r} =
    \begin{cases}
      -\frac{1}{2} c_{\{i,j\},r} + \bar{\beta}_{\{i,j\},r}, & \text{if } i < j,\\
      -\frac{1}{2} c_{\{i,j\},r} - \bar{\beta}_{\{i,j\},r}, & \text{otherwise.}
    \end{cases}
  \]
  Since this problem can be solved in polynomial time, the assertion
  follows by solving this problem for each of the~$n$ teams.
\end{proof}

\section{Comparing the strength of the different formulations}
\label{sec:strength}

In the previous section, we have introduced three different models for
finding an optimal schedule for problem SRR.
While the traditional formulation contains both polynomially many variables
and constraints, the matching and permutation formulation make use of an
exponential number of variables.
The aim of this section is to investigate whether the increase in the
number of variables in comparison with the traditional formulation leads to
a stronger formulation.
We measure the strength of a formulation based on the value of its LP
relaxation, where a higher value of the LP relaxation indicates a stronger
formulation as the LP relaxation's value is closer to the optimum value of
the integer program, as encapsulated by the following definitions.

\begin{definition}
  Let~$f$ and~$g$ be mixed-integer programming formulations of the SRR
  problem and denote by~$\lprelaxation{f}(I)$ and~$\lprelaxation{g}(I)$ the
  value of the respective LP relaxations for an instance~$I$ of~SRR.
  \begin{itemize}
  \item We say that~$f$ and~$g$ are \emph{relaxation-equivalent}
    if, for each instance $I$ of problem SRR,
    the value of the linear programming relaxations are equal, i.e.,
    $\lprelaxation{f}(I) = \lprelaxation{g}(I)$.

  \item We say that~$f$ is stronger than (or dominates) $g$ if
    \begin{enumerate*}[label=(\roman*), ref=(\roman*)]
    \item for each instance $I$ of problem SRR,
      $\lprelaxation{f}(I) \geq \lprelaxation{g}(I)$,
      and
    \item there exists an instance $I$ of problem SRR
      for which~$\lprelaxation{f}(I) > \lprelaxation{g}(I)$.
    \end{enumerate*}
  \end{itemize}
\end{definition}
We now proceed by formally comparing the strength of the formulations from
Section~\ref{sec:models} using the terminology of these definitions.
We state our results using three lemmata, and summarize all our results in Theorem~\ref{th:overallstrength}.

First, we show that the traditional and permutation formulation have
equivalent LP relaxations.
\begin{lemma}
  \label{lem:equitraandper}
  The permutation formulation~\eqref{per} is relaxation-equivalent to the traditional
  formulation~\eqref{tra}.
\end{lemma}
\begin{proof}
  To prove this lemma, we show that there is a one-to-one correspondence of
  feasible solutions of the traditional formulation's and the permutation
  formulation's LP relaxations that preserves the objective value.
  First, we construct a solution of the LP relaxation of the traditional
  formulation from a solution~$z$ of the LP relaxation of the permutation
  formulation.
  To this end, define for each~$\{i,j\} \in \matches$ and~$r \in \rounds$ a
  solution~$x \in \R^{\matches \times \rounds}$ via~$x_{\{i,j\},r} =
  \sum_{\pi \in \permswotor{i}{j}{r}} z_{i,\pi}$.
  Note that~$x$ is non-negative as all~$z$-variables are non-negative.
  Moreover, it is well-defined as~$\sum_{\pi \in \permswotor{i}{j}{r}}
  z_{i,\pi} = \sum_{\pi \in \permswotor{j}{i}{r}} z_{j,\pi}$ due
  to~\eqref{per:linking}.
  Finally, all constraints of type~\eqref{tra:matchplayed}
  and~\eqref{tra:teamplays} are satisfied since
  \begin{align*}
    &\text{for each } \match \in \matches: &
    \sum_{r \in \rounds} x_{\{i,j\},r}
    =
    \sum_{r \in \rounds} \sum_{\pi \in \permswotor{i}{j}{r}} z_{i,\pi}
    =
    \sum_{\pi \in \bigcup_{r \in \rounds} \permswotor{i}{j}{r}} z_{i,\pi}
    =
    \sum_{\pi \in \permswo{i}} z_{i, \pi}
    &\overset{\eqref{per:scheduleforteam}}{=}
      1,\\
    &\text{for each } i \in \teams,\; r \in \rounds: 
    &\sum_{j \in T \setminus \{i\}} x_{\{i,j\},r}
    =
    \sum_{j \in T \setminus \{i\}} \sum_{\pi \in \permswotor{i}{j}{r}}z_{i,\pi}
    =
    \sum_{\pi \in \Pi^{-i}} z_{i, \pi}
    &\overset{\eqref{per:scheduleforteam}}{=}
      1.
  \end{align*}
  We conclude the proof by constructing a feasible solution for the LP
  relaxation of the permutation formulation from a feasible solution~$x$ of
  the traditional formulation's LP relaxation.

  Let~$x$ be such a solution and let~$i \in \teams$.
  Consider the matrix~$X^i \in \R^{(\teams \setminus \{i\}) \times
    \rounds}$ with entries~$X^i_{j,r} = x_{\{i,j\},r}$.
  Due to all constraints of the traditional formulation's LP relaxation,
  $X^i$ is a doubly stochastic matrix and is thus contained in the Birkhoff
  polytope, see~\cite{Ziegler1995}.
  Consequently, $X^i$ can be written as a convex combination of all
  permutation matrices.
  That is, if~$P^{i,\pi}$ is the permutation matrix associated with~$\pi \in
  \permswo{i}$, there exist multipliers~$\lambda^i_\pi \geq 0$, $\pi \in
  \permswo{i}$, such that~$X^i = \sum_{\pi \in \permswo{i}} \lambda^i_\pi
  P^{i,\pi}$ and~$\sum_{\pi \in \permswo{i}} \lambda^i_\pi = 1$.
  Based on these multipliers, we define a solution~$z$ of the permutation
  formulation via~$z_{i,\pi} = \lambda^i_{\pi}$.
  To conclude the proof, we need to show that this solution~$z$ is feasible
  for the permutation formulation's LP relaxation and has the same
  objective value as~$x$.
  Observe that~$z$ is non-negative since all~$\lambda$'s are non-negative.
  Constraints~\eqref{per:scheduleforteam} and~\eqref{per:linking} are
  satisfied as
  \begin{align*}
    &\text{for each } i \in \teams:&
    \sum_{\pi\in\permswo{i}} z_{i,\pi}
    &=
    \sum_{\pi\in\permswo{i}} \lambda^i_\pi
    =
    1,\\
    &\text{for each } \{i,j\} \in\matches,\; r \in \rounds:&
    \sum_{\pi \in \permswotor{i}{j}{r}} z_{i,\pi}
    &=
    \sum_{\pi \in \permswotor{i}{j}{r}} \lambda^i_\pi
    =
    x_{\{i,j\},r}
    =
    \sum_{\pi \in \permswotor{j}{i}{r}} \lambda^j_\pi
    =
    \sum_{\pi\in\permswo{j}} z_{j,\pi}
  \end{align*}
  since~$\sum_{\pi\in\permswo{i}} \lambda^i_\pi = 1$ and~$x_{\{i,j\},r}$ is a
  convex combination of permutation matrices that assign team~$j$ (or~$i$)
  to round~$r$, respectively.
  Consequently, $z$ is feasible for the permutation formulation's LP relaxation.
  Finally, both~$x$ and~$z$ have the same objective value because
  \begin{align*}
    \frac{1}{2} \sum_{i \in\teams} \sum_{\pi \in \permswo{i}} e_{i,\pi}z_{i,\pi}
    &=
    \frac{1}{2} \sum_{i \in\teams} \sum_{\pi \in \permswo{i}} e_{i,\pi}\lambda^i_\pi
    =
    \frac{1}{2} \sum_{i \in\teams} \sum_{\pi \in \permswo{i}} \sum_{r \in \rounds} c_{\{i,\pi_r\},r}\lambda^i_\pi\\
    &=
    \frac{1}{2} \sum_{i \in\teams} \sum_{j \in \teams\setminus\{i\}} \sum_{r \in \rounds} c_{\{i,j\},r}
      \sum_{\substack{\pi\in\permswo{i}\colon\\ \pi_r = j}} \lambda^i_\pi P^{i,\pi}_{\pi_r,r}
    =
    \frac{1}{2} \sum_{i \in\teams} \sum_{j \in \teams\setminus\{i\}} \sum_{r \in \rounds} c_{\{i,j\},r} x_{\{i,j\},r}\\
    &=
    \sum_{\{i,j\} \in \matches}\sum_{r \in \rounds} c_{\{i,j\},r}x_{\{i,j\},r}.
  \end{align*}
  which proves that both formulations are relaxation-equivalent.
\end{proof}
Next, we turn our focus to the matching formulation and compare it with the
traditional formulation (and thus, by the previous lemma, also with the
permutation formulation).
\begin{lemma}
  \label{lem:strength}
  For each $n \geq 6$, the matching
  formulation~\eqref{mat} is stronger than
  the traditional formulation~\eqref{tra}.
\end{lemma}
\begin{proof}
  First, we show that we can transform any feasible solution of the
  matching formulation's LP relaxation to a feasible solution of the
  traditional formulation's LP relaxations.
  Afterwards, to show that the matching formulation is stronger than the
  traditional formulation, we show that, for any even~$n \geq 6$, there
  exists an instance of SRR for which the LP relaxation of the matching
  formulation has a strictly larger value than the traditional
  formulation's LP relaxation.

  Let~$y$ be a feasible solution of the matching formulation's LP
  relaxation.
  We construct a solution~$x$ for the traditional formulation by
  setting~$x_{\match,r} = \sum_{\matching\in\matchings\colon
    \match\in\matching} y_{\match,r}$.
  Since~$y$ is non-negative, also~$x$ is non-negative.
  Moreover, Conditions~\eqref{tra:matchplayed} and~\eqref{tra:teamplays}
  are satisfied as
  \begin{align*}
    &\text{for each } \match\in\matches:&
    \sum_{r \in \rounds} x_{\match,r}
    &=
    \sum_{r \in \rounds} \sum_{\substack{\matching\in\matchings\colon\\ \match\in\matching}} y_{\matching,r}
    \overset{~\eqref{mat:matchplayed}}{=}
    1,\\
    &\text{for each } i\in\teams,\; r \in\rounds:&
    \sum_{j \in T \setminus \{i\}} x_{\{i,j\},r}
    &=
    \sum_{j \in T \setminus \{i\}} \sum_{\substack{\matching\in\matchings\colon\\ \{i,j\}\in\matching}} y_{\{i,j\},r}
    =
    \sum_{\matching\in\matchings} y_{\matching,r}
    \overset{~\eqref{mat:matchingonround}}{=}
    1.
  \end{align*}
  Finally, both~$x$ and~$y$ have the same objective value as
  \[
    \sum_{\match\in\matches}\sum_{r\in\rounds} c_{\match,r}x_{\match,r}
    =
    \sum_{\match\in\matches}\sum_{r\in\rounds} c_{\match,r}\sum_{\substack{\matching\in\matchings\colon\\ \match\in\matching}} y_{\matching,r}
    =
    \sum_{\matching\in\matchings}\sum_{r\in\rounds} \sum_{\match\in\matching}c_{\match,r} y_{\matching,r}
    =
    \sum_{\matching\in\matchings}\sum_{r\in\rounds} d_{\matching,r} y_{\matching,r},
  \]
  that is, the traditional formulation cannot be stronger than the matching
  formulation.

  To prove that the matching formulation dominates the traditional
  formulation for~$n\geq 6$ even, we distinguish three cases.
  In the first case, assume~$n \geq 10$.
  Consider the pairs of teams given by
  \[
    P
    =
    \big\{\{1,2\}, \{2,3\}, \{1,3\}\big\}
    \cup
    \big\{\{4,5\}, \{5,6\}, \{4,6\}\big\}
    \cup
    \big\{\{7,8\},\{8,9\},\dots,\{n-1,n\},\{7,n\}\big\}.
  \]
  Interpreting~$P$ as the edges of an undirected graph, $P$ defines three
  connected components consisting of two 3-cycles and an even cycle.
  We construct an instance of the SRR problem by specifying the cost
  function~$c \in \R^{\matches\times\rounds}$ via
  \[
    c_{\match,r}
    =
    \begin{cases}
      1, & \text{if } \match \notin P \text{ and } r \in \{1,2\},\\
      0, & \text{otherwise.}
    \end{cases}
  \]
  It is easy to verify that~$x \in \R^{\matches \times \rounds}$ given by
  \[
    x_{\match,r}
    =
    \begin{cases}
      \tfrac{1}{2}, & \text{if } \match \in P \text{ and } r \in \{1,2\},\\
      0, & \text{if } \match \notin P \text{ and } r \in \{1,2\},\\
      \tfrac{1}{n-3}, & \text{otherwise},
    \end{cases}
  \]
  is feasible for the LP relaxation of the traditional formulation and has
  objective value~0.
  Hence, $x$ is optimal.

  Solving the LP relaxation of the matching formulation for this instance,
  however, results in an objective value that is at least~2.
  Indeed, each perfect matching $\matching\in\matchings$
  contains at least one match $\match \in \matching$
  with $\match \in \{\{i,j\} : (i,j) \in\{1, 2, 3 \} \times \{4, 5, \dots, n\}\}$.
  Since $c_{\match, 1} = c_{\match,2} = 1$ for such a match, it follows
  that in both rounds~1 and~2, matchings are selected with total weight at
  least 1 due to~\eqref{mat:matchplayed},leading to a solution with total
  cost at least 2.

  In the second case, we consider~$n = 6$.
  To prove the statement, we use the same construction as before, however,
  we do not require the even cycle anymore.
  That is, $P$ defines two 3-cycles and the argumentation remains the same
  as before.

  In the last case~$n = 8$, we consider the set of pairs
  \[
    P
    =
    \big\{\{1,2\}, \{1,3\}, \{2,3\}\big\} \cup \big\{ \{4,5\}, \{5,6\}, \{6,7\},
    \{7,8\}, \{4,8\}\big\}.
  \]
  If we interpret~$P$ as edges of an undirected graph, the corresponding
  graph has two connected components being a~3-cycle and a~5-cycle,
  respectively.
  We choose the cost-coefficients~$c \in \R^{\matches \times \rounds}$ to be
  \[
    c_{\match, r} =
    \begin{cases}
      1, & \text{if } \match \notin P \text{ and } r \in \{1,2\},\\
      0, & \text{otherwise.}
    \end{cases}
  \]
  Simple calculations show that an optimal solution of the traditional
  formulation's LP relaxation is~$x \in \R^{\matches \times \rounds}$ with
  \[
    x_{\match, r}
    =
    \begin{cases}
      \tfrac{1}{2}, & \text{if } \match \in P \text{ and } r\in \{1,2\},\\
      0, & \text{if } \match \notin E \text{ and } r\in\{1,2\},\\
      \tfrac{1}{5}, & \text{otherwise,}
    \end{cases}
  \]
  which has objective value~0, whereas the matching formulation's LP
  relaxation has value~2.
\end{proof}
The previous two lemmata completely characterize the relative strength of
the three different formulations except for~$n=4$.
The status of this missing case is settled next.
\begin{lemma}
  \label{lem:n=4}
  For $n=4$, the traditional formulation and the matching formulation are
  relaxation-equivalent.
\end{lemma}
\begin{proof}
  Observe that exactly the same arguments as in the proof of
  Lemma~\ref{lem:strength} can be used to show that the matching
  formulation is at least as strong as the traditional formulation
  if~$n=4$.
  Hence, it remains to show that every solution~$x$ of the traditional
  formulation's LP relaxation can be turned into a solution of the LP
  relaxation of the matching formulation if~$n=4$.
  Let~$x$ be such a solution.
  Let~$\matching = \big\{\{i,j\}, \{k,l\}\big\} \in \matchings$.
  Since~$x$ satisfies Equations~\eqref{tra:teamplays}, summing the
  equations for~$i,j$ and subtracting the equations for~$k,l$
  yields~$2x_{\{i,j\},r} - 2x_{\{k,l\},r} = 0$.
  That is, the~$x$-variables for the two matches of a matching within the
  same round have the same value.
  Consequently, the solution~$y \in \R^{\matchings \times \rounds}$
  given by~$y_{\matching,r} = x_{\{i,j\},r}$ is well-defined and it is
  immediate to check that~$y$ is feasible for the matching formulation's LP
  relaxation and has the same objective value as~$x$.
\end{proof}
Summarizing the previous results of this section, we can provide a complete
comparison of the strength of the traditional, matching, and permutation
formulation.
\begin{theorem}
  \label{th:overallstrength}
  For each~$n\geq 6$, the traditional and permutation formulation are
  relaxation-equivalent, whereas the matching formulation is stronger than either of them.
  For~$n=4$, the traditional, matching, and permutation formulation for problem SRR are relaxation-equivalent.
\end{theorem}
Besides verifying that all three models are equivalent for~$n=4$, we can
also show that the matching formulation's integer hull is already
completely characterized by~\eqref{mat:matchingonround},
\eqref{mat:matchplayed}, as well as non-negativity inequalities for all
variables.

\begin{proposition}
  For~$n=4$,
  Equations~\eqref{mat:matchingonround} and~\eqref{mat:matchplayed} as well
  as non-negativity inequalities define an integral polyhedron.
  That is, the matching formulation's LP relaxation coincides with its
  integer hull.
\end{proposition}

\begin{proof}
  To prove the proposition's statement, we show that the constraint matrix
  of~\eqref{mat} is totally unimodular.
  The result follows then by the Hoffman-Kruskal
  theorem~\cite{Schrijver1987} as all right-hand
  side values in~\eqref{mat} are integral.

  For $n=4$, the set of all matchings $\matchings$ consists of exactly the
  three matchings
  \begin{align*}
    \matching_1
    &= \big\{ \{1, 2\}, \{3, 4\} \big\},
    &
      \matching_2
    &= \big\{ \{1, 3\}, \{2, 4\} \big\},
    &
      \matching_3 &= \big\{ \{1, 4\}, \{2, 3\} \big\}.
  \end{align*}
  The non-trivial constraints from Formulation~\eqref{mat} are
  Equations~\eqref{mat:matchingonround} and~\eqref{mat:matchplayed}, which
  yield system
  \[
    \begin{pmatrix}
      1& & &1& & &1& &  \\
      &1& & &1& & &1&  \\
      & &1& & &1& & &1 \\
      1&1&1& & & & & &  \\
      & & &1&1&1& & &  \\
      & & & & & &1&1&1 \\
      & & & & & &1&1&1 \\
      & & &1&1&1& & &  \\
      1&1&1& & & & & &  \\
    \end{pmatrix}
    \begin{pmatrix}
      y_{\matching_1, 1} \\
      y_{\matching_1, 2} \\
      y_{\matching_1, 3} \\
      y_{\matching_2, 1} \\
      y_{\matching_2, 2} \\
      y_{\matching_2, 3} \\
      y_{\matching_3, 1} \\
      y_{\matching_3, 2} \\
      y_{\matching_3, 3} \\
    \end{pmatrix}
    =
    \begin{pmatrix}
      1\\1\\1\\1\\1\\1\\1\\1\\1
    \end{pmatrix}
    .
    \quad
    \begin{array}{l}
      (\eqref{mat:matchingonround}, r=1) \\
      (\eqref{mat:matchingonround}, r=2) \\
      (\eqref{mat:matchingonround}, r=3) \\
      (\eqref{mat:matchplayed}, m=\{1, 2\}) \\
      (\eqref{mat:matchplayed}, m=\{1, 3\}) \\
      (\eqref{mat:matchplayed}, m=\{1, 4\}) \\
      (\eqref{mat:matchplayed}, m=\{2, 3\}) \\
      (\eqref{mat:matchplayed}, m=\{2, 4\}) \\
      (\eqref{mat:matchplayed}, m=\{3, 4\}) \\
    \end{array}
  \]
  Note that the last three equations are redundant and can be removed.
  The constraint matrix of the remaining equations is the node-edge
  incidence matrix of a bipartite graph and hence totally unimodular, which
  concludes the proof.
\end{proof}

Thus, for $n=4$, simply solving the LP-relaxation of the matching formulation by the simplex method, suffices to find an optimum integral solution.

\section{Strengthening the formulations}
\label{sec:inequalities}

In this section, we continue our investigations of the structure of the formulations. In Section~\ref{sec:matchingcontinued}, we derive an exponentially sized class of valid inequalities for the matching formulation. Also, we show in Section~\ref{sec:tracontinued} that adding the so-called odd-cut inequalities to the traditional formulation yields a formulation that is relaxation-equivalent to the matching formulation.

\subsection{Strengthening the matching formulation}
\label{sec:matchingcontinued}

Observe that Theorem~\ref{th:overallstrength} does not rule out the
possibility that, for $n\geq 6$, every vertex of the matching formulation
is integral.
That, however, is not the case already for~$n = 6$ as we will show next.
To this end, we first provide a fractional point~$y^\star$ that is contained in the LP
relaxation of the matching formulation for~$n=6$.
Afterwards, we derive a class of valid inequalities for the integer hull matching
formulation, and finally, we provide one such inequality that
is violated by~$y^\star$.

\begin{example}
  \label{ex:examplesolution}
  Let $n=6$.
  Then, the set of teams and rounds is given by~$\teams = \{1, \dots, 6\}$
  and rounds~$\rounds = \{1, \dots, 5\}$, respectively.
  In Figure~\ref{fig:examplesolution}, we depict a fractional solution of
  the matching formulation's LP relaxation.
  For each round~$r \in \rounds$, we provide two perfect matchings between the
  teams~$T$, the blue and green (dashed) matching~$\matching$, whose
  corresponding variables~$y_{\matching,r}$ have value~$\frac{1}{2}$ in the
  corresponding solution; all remaining variables have value~0.
  It is easy to verify that this fractional solution is indeed feasible for
  the LP relaxation of~\eqref{mat}.

  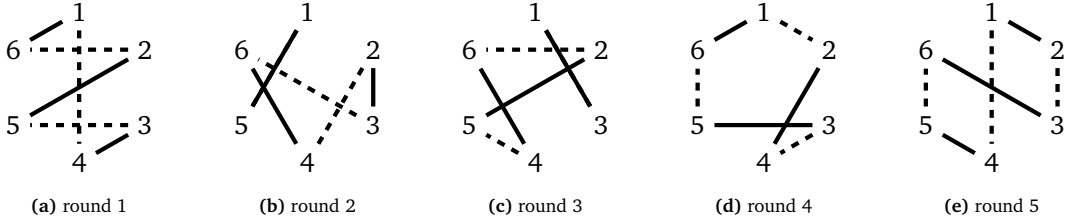
\begin{figure}[ht]
    \begin{subfigure}{\textwidth/5}
      \centering
      \begin{tikzpicture}
        \node (0) at ({sin(0*360/6)}, {cos(0*360/6)}) {1};
        \node (1) at ({sin(1*360/6)}, {cos(1*360/6)}) {2};
        \node (2) at ({sin(2*360/6)}, {cos(2*360/6)}) {3};
        \node (3) at ({sin(3*360/6)}, {cos(3*360/6)}) {4};
        \node (4) at ({sin(4*360/6)}, {cos(4*360/6)}) {5};
        \node (5) at ({sin(5*360/6)}, {cos(5*360/6)}) {6};

        \draw[draw=darkblue, ultra thick] (2) edge (3);
        \draw[draw=darkblue, ultra thick] (1) edge (4);
        \draw[draw=darkblue, ultra thick] (0) edge (5);

        \draw[draw=darkgreen, dashed, ultra thick] (2) edge (4);
        \draw[draw=darkgreen, dashed, ultra thick] (0) edge (3);
        \draw[draw=darkgreen, dashed, ultra thick] (1) edge (5);
      \end{tikzpicture}
      \caption{round 1}
    \end{subfigure}%
    \begin{subfigure}{\textwidth/5}
      \centering
      \begin{tikzpicture}
        \node (0) at ({sin(0*360/6)}, {cos(0*360/6)}) {1};
        \node (1) at ({sin(1*360/6)}, {cos(1*360/6)}) {2};
        \node (2) at ({sin(2*360/6)}, {cos(2*360/6)}) {3};
        \node (3) at ({sin(3*360/6)}, {cos(3*360/6)}) {4};
        \node (4) at ({sin(4*360/6)}, {cos(4*360/6)}) {5};
        \node (5) at ({sin(5*360/6)}, {cos(5*360/6)}) {6};

        \draw[draw=darkblue, ultra thick] (1) edge (2);
        \draw[draw=darkblue, ultra thick] (0) edge (4);
        \draw[draw=darkblue, ultra thick] (3) edge (5);

        \draw[draw=darkgreen, dashed, ultra thick] (2) edge (5);
        \draw[draw=darkgreen, dashed, ultra thick] (1) edge (3);
        \draw[draw=darkgreen, dashed, ultra thick] (0) edge (4);
      \end{tikzpicture}
      \caption{round 2}
    \end{subfigure}%
    \begin{subfigure}{\textwidth/5}
      \centering
      \begin{tikzpicture}
        \node (0) at ({sin(0*360/6)}, {cos(0*360/6)}) {1};
        \node (1) at ({sin(1*360/6)}, {cos(1*360/6)}) {2};
        \node (2) at ({sin(2*360/6)}, {cos(2*360/6)}) {3};
        \node (3) at ({sin(3*360/6)}, {cos(3*360/6)}) {4};
        \node (4) at ({sin(4*360/6)}, {cos(4*360/6)}) {5};
        \node (5) at ({sin(5*360/6)}, {cos(5*360/6)}) {6};

        \draw[draw=darkblue, ultra thick] (0) edge (2);
        \draw[draw=darkblue, ultra thick] (1) edge (4);
        \draw[draw=darkblue, ultra thick] (3) edge (5);

        \draw[draw=darkgreen, dashed, ultra thick] (0) edge (2);
        \draw[draw=darkgreen, dashed, ultra thick] (3) edge (4);
        \draw[draw=darkgreen, dashed, ultra thick] (1) edge (5);
      \end{tikzpicture}
      \caption{round 3}
    \end{subfigure}%
    \begin{subfigure}{\textwidth/5}
      \centering
      \begin{tikzpicture}
        \node (0) at ({sin(0*360/6)}, {cos(0*360/6)}) {1};
        \node (1) at ({sin(1*360/6)}, {cos(1*360/6)}) {2};
        \node (2) at ({sin(2*360/6)}, {cos(2*360/6)}) {3};
        \node (3) at ({sin(3*360/6)}, {cos(3*360/6)}) {4};
        \node (4) at ({sin(4*360/6)}, {cos(4*360/6)}) {5};
        \node (5) at ({sin(5*360/6)}, {cos(5*360/6)}) {6};

        \draw[draw=darkblue, ultra thick] (2) edge (4);
        \draw[draw=darkblue, ultra thick] (1) edge (3);
        \draw[draw=darkblue, ultra thick] (0) edge (5);

        \draw[draw=darkgreen, dashed, ultra thick] (0) edge (1);
        \draw[draw=darkgreen, dashed, ultra thick] (4) edge (5);
        \draw[draw=darkgreen, dashed, ultra thick] (2) edge (3);
      \end{tikzpicture}
      \caption{round 4}
    \end{subfigure}%
    \begin{subfigure}{\textwidth/5}
      \centering
      \begin{tikzpicture}
        \node (0) at ({sin(0*360/6)}, {cos(0*360/6)}) {1};
        \node (1) at ({sin(1*360/6)}, {cos(1*360/6)}) {2};
        \node (2) at ({sin(2*360/6)}, {cos(2*360/6)}) {3};
        \node (3) at ({sin(3*360/6)}, {cos(3*360/6)}) {4};
        \node (4) at ({sin(4*360/6)}, {cos(4*360/6)}) {5};
        \node (5) at ({sin(5*360/6)}, {cos(5*360/6)}) {6};

        \draw[draw=darkblue, ultra thick] (0) edge (1);
        \draw[draw=darkblue, ultra thick] (2) edge (5);
        \draw[draw=darkblue, ultra thick] (3) edge (4);

        \draw[draw=darkgreen, dashed, ultra thick] (4) edge (5);
        \draw[draw=darkgreen, dashed, ultra thick] (1) edge (2);
        \draw[draw=darkgreen, dashed, ultra thick] (0) edge (3);
      \end{tikzpicture}
      \caption{round 5}
    \end{subfigure}%
    \caption{A feasible point for the LP relaxation of Formulation~\eqref{mat}.}
    \label{fig:examplesolution}
  \end{figure}
\end{example}

To describe our class of valid inequalities, consider the following lemma.

\begin{lemma}
  \label{lem:simpleCGcut}
  Let~$\match_1, \match_2 \in \matches$ be disjoint and let~$r' \in
  \rounds$.
  Then,
  \begin{equation}
    \label{eq:cgcutmat}
    \sum_{r \in \rounds\setminus\{r'\}}\sum_{\substack{\matching\in\matchings\colon\\
    \match_1 \in \matching \text{ or } \match_2\in\matching}}
    y_{\matching,r}
    +
    \sum_{\substack{\matching\in\matchings\colon\\
    \match_1 \notin \matching \text{ or } \match_2\notin\matching}}
    y_{\matching,r'}
    +
    \sum_{\substack{\matching\in\matchings\colon\\
    \match_1, \match_2 \in \matching}}
    2y_{\matching,r'}
    \geq 2
  \end{equation}
  is a valid inequality for~\eqref{mat}.
  In particular, it is a Chv\'atal-Gomory cut derived from the LP
  relaxation of~\eqref{mat}.
\end{lemma}
\begin{proof}
  It is sufficient to prove that~\eqref{eq:cgcutmat} is indeed a
  Chv\'atal-Gomory cut.
  To this end, we multiply Equation~\eqref{mat:matchingonround} for round
  $r'$ and Equations~\eqref{mat:matchplayed} for matches~$\match_1$
  and~$\match_2$ by~$\frac{1}{2}$ and sum the resulting equations to obtain
  \[
    \sum_{r \in \rounds\setminus\{r'\}}\sum_{\substack{\matching\in\matchings\colon\\
    \match_1 \in \matching \text{ or } \match_2\in\matching}}
    \frac{C_{\matching,r}}{2} y_{\matching,r}
    +
    \sum_{\substack{\matching\in\matchings\colon\\
    \match_1 \notin \matching \text{ or } \match_2\notin\matching}}
    \frac{1 + C_{\matching,r}}{2}y_{\matching,r'}
    +
    \sum_{\substack{\matching\in\matchings\colon\\
    \match_1, \match_2 \in \matching}}
    \tfrac{3}{2}y_{\matching,r'}
    = \frac{3}{2},
  \]
  where~$C_{\matching,r} = \card{\matching \cap \{\match_1,\match_2\}}$.
  Since all~$y$-variables are non-negative, we can turn this equation into
  a~$\geq$-inequality by rounding up the left-hand side coefficients.
  Moreover, since in a feasible solution for~\eqref{mat} all variables
  attain integer values, we can increase the right-hand side
  from~$\frac{3}{2}$ to~2, which yields the desired inequality.
\end{proof}
Using this class of inequalities, we can show that the point~$y^\star$
presented in the previous example is indeed not contained in the matching
formulation's integer hull.
Select~$r' = 1$, $\match_1 = \{1,6\}$, and~$\match_2 = \{3,5\}$, and
let~$\matching$ be the blue and~$\matching'$ be the green matching of the
first round as well as~$\matching''$ the blue matching of round~$4$.
Then, the corresponding inequality's left-hand side evaluates in~$y^\star$ to
$y^\star_{M'',4} + y^\star_{M,1} + y^\star_{M',1} = \frac{3}{2}$.
Hence, $y^\star$ violates the corresponding inequality as~$\frac{3}{2} \ngeq 2$.

Note that Inequality~\eqref{eq:cgcutmat} is a
so-called~$\{0, \frac{1}{2}\}$-cut~\cite{CapraraFischetti1996} as all multipliers used in the
derivation are~$\frac{1}{2}$ (and~0 for inequalities/equations that have
not been used).
By taking more equations in the generation of a valid inequality into
account, we can generalize~\eqref{eq:cgcutmat} to an exponentially large
class of inequalities.
\begin{proposition}
  Let~$A \subseteq \matches$ be a set of pairwise disjoint matches and
  let~$B \subseteq \rounds$.
  If~$\card{A} + \card{B}$ is odd, then
  \begin{equation}
    \label{eq:generalcg}
    \sum_{\matching\in\matchings} \sum_{r \in B}
    \left\lceil \frac{1 + \card{\matching \cap A}}{2}\right\rceil y_{\matching,r}
    +
    \sum_{\matching\in\matchings} \sum_{r \in \rounds\setminus B}
    \left\lceil \frac{\card{\matching \cap A}}{2}\right\rceil y_{\matching,r}
    \geq
    \frac{1 + \card{A} + \card{B}}{2},
  \end{equation}
  is a valid inequality for~\eqref{mat}.
  In particular, it is a Chv\'atal-Gomory cut derived from the LP
  relaxation of~\eqref{mat}.
\end{proposition}
\begin{proof}
  We follow the line of the proof of Lemma~\ref{lem:simpleCGcut} and
  multiply each constraint of type~\eqref{mat:matchingonround} with index
  in~$A$ and each constraint of type~\eqref{mat:matchplayed} with index
  in~$B$ by~$\frac{1}{2}$ and sum all resulting equations.
  This leads to
  \[
    \sum_{j = 1}^{\nicefrac{n}{2}}
    \sum_{\substack{\matching\in\matchings\colon\\ \card{\matching\cap A} = j}}
    \sum_{r \in B}
    \frac{1 + j}{2} y_{\matching,r}
    +
    \sum_{j = 1}^{\nicefrac{n}{2}}
    \sum_{\substack{\matching\in\matchings\colon\\ \card{\matching\cap A} = j}}
    \sum_{r \in \rounds\setminus B}
    \frac{j}{2} y_{\matching,r}
    =
    \frac{\card{A} + \card{B}}{2}.
  \]
  Since all~$y$-variables are non-negative, we derive the inequality
  \[
    \sum_{\matching\in\matchings} \sum_{r \in B}
    \left\lceil \frac{1 + \card{\matching \cap A}}{2}\right\rceil y_{\matching,r}
    +
    \sum_{\matching\in\matchings} \sum_{r \in \rounds\setminus B}
    \left\lceil \frac{\card{\matching \cap A}}{2}\right\rceil y_{\matching,r}
    \geq
    \frac{\card{A} + \card{B}}{2},
  \]
  and by integrality of the~$y$-variables, we can round up the right-hand
  side, which leads to the desired inequality.
\end{proof}
While Inequalities~\eqref{eq:cgcutmat} can trivially be separated in
polynomial time, an efficient separation algorithm for~\eqref{eq:generalcg}
is not immediate.
We leave the complexity status of separating~\eqref{eq:generalcg} open for future research.

\subsection{Strengthening the traditional formulation}
\label{sec:tracontinued}

Revisiting the proof of Lemma~\ref{lem:strength}, it becomes clear that it is possible to assign, for a fixed
round, each edge (match) of an odd cycle in~$K_n$ a weight of~$\frac{1}{2}$.
That is, the traditional formulation can assign an odd cycle of length~$k$
a weight of~$\frac{k}{2}$.
Such a solution, however, cannot be written as a convex combination of
integer feasible solutions, because each such solution defines a perfect
matching on the matches of a fixed round, i.e., the total weight of an odd
cycle can be at most~$\frac{k-1}{2}$.
To strengthen the traditional formulation, one can thus add facet defining
inequalities for the perfect matching polytope~$P_M$ to Model~\eqref{tra}, which
results in the additional inequalities
\begin{align}
  \label{eq:blossom}
  \sum_{i \in U} \sum_{j \in T \setminus U} x_{\{i,j\},r} &\geq 1, && U \subseteq T \text{ with } \card{U}\text{ odd},
                                                                      r \in \rounds,
\end{align}
which correspond to the odd-cut inequalities for the matching polytope and can be
separated in polynomial time.
\begin{lemma}
  Let~$n \geq 6$.
  The traditional formulation~\eqref{tra} extended by~\eqref{eq:blossom} is
  relaxation-equivalent to the matching formulation.
\end{lemma}
\begin{proof}
  We use the same proof strategy as for Lemma~\ref{lem:strength}.
  Therefore, consider again the solution~$x \in \R^{\matches \times
    \rounds}$ given by~$x_{\match,r} = \sum_{\matching\in\matchings\colon
    \match\in\matching} y_{\matching,r}$ for a solution~$y$ of the matching
  formulation's LP relaxation.
  Due to the proof of Lemma~\ref{lem:strength}, it is sufficient to show
  that~$x$ satisfies~\eqref{eq:blossom} to prove that the matching
  formulation is at least as strong as the enhanced traditional
  formulation.
  Let~$U \subseteq \teams$ have odd cardinality.
  Since every $\matching \in \matchings$ is a perfect matching,
  there is at least one team~$i \in U$ that does not play against another
  team in $U$ since~$U$ is odd.
  Hence, for each~$\matching \in \matchings$, there is a match~$\{i,j\} \in
  \matching$ with~$i \in U$ and~$j \notin U$.
  Then,
  \begin{align*}
    \sum_{i \in U} \sum_{j \in \teams \setminus U} x_{\{i,j\},r}
    &=
    \sum_{i \in U} \sum_{j \in \teams \setminus U}
    \sum_{\substack{\matching \in \matchings\colon\\ \{i,j\} \in \matching}} y_{\matching,r}\\
    &=
      \sum_{\matching\in\matchings} \sum_{i \in U} \sum_{\substack{j \in T
      \setminus U\colon\\ \{i,j\}\in \matching}} y_{\matching,r}
    \geq
    \sum_{\matching\in\matchings} y_{\matching,r}
    \overset{~\eqref{mat:matchplayed}}{\geq}
    1.
  \end{align*}
  Consequently, the matching formulation is at least as strong as the
  enhanced traditional formulation.

  To prove that the enhanced traditional formulation is not weaker than the matching
  formulation, we use a strategy similar to the one pursued in the proof of
  Lemma~\ref{lem:equitraandper}.
  Since the enhanced traditional formulation contains, per round~$r$, all facet
  defining inequalities as well as equations for the perfect matching
  polytope~$P_M$, each vector~$X^r \in \R^{\matches}$ given
  by~$X^r_{\{i,j\}} = x_{\{i,j\},r}$ is contained in~$P_M$.
  Hence, there exist non-negative multipliers~$\lambda^r \in \R_+^{\matchings}$
  with~$\sum_{\matching \in \matchings} \lambda^r_{\matching} = 1$ such
  that~$X^r = \sum_{\matching\in\matchings} \lambda^r_\matching
  V_\matching$, where~$V_\matching$ is the vertex of~$P_M$ corresponding to
  the perfect matching~$\matching$.
  We claim that~$y \in \R^{\matchings \times \rounds}$ given
  by~$y_{\matching,r} = \lambda^r_\matching$ is feasible for the LP
  relaxation of the matching formulation.
  Because~$X^r = \sum_{\matching \in \matchings} \lambda^r_\matching
  V_\matching$ implies~$X^r_{\match} = \sum_{\matching\in\matchings\colon
    \match\in\matching} \lambda^r_\matching$,
  both~\eqref{mat:matchingonround} and~\eqref{mat:matchplayed} are
  satisfied as
  \begin{align*}
    \sum_{\matching \in \matchings} y_{\matching,r}
    &=
      \sum_{\matching \in \matchings} \lambda^r_\matching
      =
      1,\\
    \sum_{\substack{\matching\in\matchings\colon\\ \match \in \matching}}
    \sum_{r \in \rounds} y_{\matching,r}
    &=
    \sum_{\substack{\matching\in\matchings\colon\\ \match \in \matching}}
    \sum_{r \in \rounds} \lambda^r_\matching
    =
    \sum_{r \in \rounds} x_{\match,r}
    \overset{~\eqref{tra:matchplayed}}{=}
    1.
  \end{align*}
  Moreover, $y$ is non-negative as the~$\lambda$'s form a convex
  combination and both~$x$ and~$y$ have the same objective value since
  \[
    \sum_{\matching\in\matchings}\sum_{r\in\rounds} d_{\matching,r}y_{\matching,r}
    =
    \sum_{\matching\in\matchings}\sum_{r\in\rounds} \sum_{\match\in\matching}c_{\match,r}\lambda^r_\matching
    =
    \sum_{\match\in\matches}\sum_{r\in\rounds}\sum_{\substack{\matching\in\matchings\colon\\
        \match\in\matching}} c_{\match,r}\lambda^r_\matching
    =
    \sum_{\match\in\matches}\sum_{r\in\rounds} c_{\match,r}x_{\match,r},
  \]
  which concludes the proof.
\end{proof}
\begin{remark}
  Since the traditional and permutation formulation are equivalent, one
  might wonder whether also the permutation formulation can be enhanced by
  odd-cut inequalities.
  Indeed, using the transformation~$x_{\{i,j\},r} = \sum_{\pi \in
    \permswotor{i}{j}{r}} z_{i,\pi}$ as in the proof of
  Lemma~\ref{lem:equitraandper}, one can show that the corresponding
  version of odd-cut inequalities is given by
  \begin{align*}
    \sum_{i \in U} \sum_{j \in \teams\setminus U} \sum_{\pi \in
    \permswotor{i}{j}{r}} z_{i,\pi} &\geq 1,
    && U \subseteq T \text{ with } \card{U}\text{ odd}, r \in \rounds,
  \end{align*}
  and that the enhanced traditional and permutation formulation are
  equivalent.
\end{remark}

\section{An extension: $k$-round robin tournaments}
\label{sec:kRR}

In this section, we generalize the models for single round robin
tournaments to $k$-round robin tournaments, where each pair of teams is
required to meet exactly $k$ times, for $k \geq 1$. As a consequence, the total number of matches that need to be scheduled
becomes $\frac12kn(n-1)$, and we set $\rounds \define \{1,2, \ldots, k(n-1)\}$.

\begin{problem}[$k$RR]
  \label{prob:krr}
  Let~$n \geq 4$ be an even integer and let~$k \geq 1$ be integral.
  Given~$n$ teams with corresponding matches~$\matches$,
  a set of~$\frac12kn(n-1)$ rounds~$\rounds$ ($k \geq 1$), as well as an integral cost~$c_{\match,r}$ for
  every~$\match\in\matches$ and round~$r \in \rounds$, the $k$-\emph{round robin ($k$RR)} problem is to find an assignment~$\mathcal{A}
  \subseteq \matches \times \rounds$ of matches to rounds such that (i) every
  team plays a single match per round, (ii) each match is played in $k$, pairwise distinct, rounds, while total cost $\sum_{(\match,r) \in \mathcal{A}} c_{\match,r}$ is minimized.
\end{problem}
\noindent
Problem~\ref{prob:srr} (SRR) is a special case of $k$RR as it arises when $k = 1$.
Another very prominent special case arises
when $k=2$, the so-called Double Round Robin tournament, denoted hereafter
by DRR.

In principle, it is easy to generalize the models from
Section~\ref{sec:models} to account for meeting $k$ times instead of once.
Indeed, by replacing the
right-hand side of constraints \eqref{tra:matchplayed}, or the right-hand
side of constraints \eqref{mat:matchplayed} by $k$, or by redefining
$\permswo{i}$ and $\permswotor{i}{j}{r}$ to ordered lists
for team $i$ that features every opponent $j$ exactly $k$ times,
the resulting formulations for $k$RR directly arise. In fact, we
claim that it is not difficult to verify that the results concerning the polynomial
solvability of the linear relaxations (Lemmata~\ref{lem:pricematch} and \ref{lem:pricepermutation}),
as well as the strength of the relaxations (Theorem~\ref{th:overallstrength}) hold for the $k$RR for each $k \geq 1$.

However, in practice, a number of additional properties become relevant when considering $k$-round robin tournaments:
\emph{phased} tournaments and tournaments where playing \emph{home} or \emph{away} matters. We now discuss these properties, and their consequences for the formulations, in more detail.

\paragraph{Phased} (PH)
The tournament is split into $k$ parts such that each pair
of teams meets once in each part. Here a {\em part} of the tournament
refers to $n-1$ consecutive rounds, starting at round $\ell(n-1) +1$,
for $\ell \in \{0, \dots, k-1\}$.
Moreover,
we use $\rounds_{\ell} \define \{\ell (n-1)+1, \dots, (\ell+1)(n-1)\}$
to denote the rounds in part $\ell \in \{0, \dots, k-1\}$,
and $\rounds \define \bigcup_{\ell=0}^{k-1} \rounds_{\ell}$.

Without the presence of any additional constraints, a phased tournament can be
trivially decomposed in multiple single-round robin tournaments: one for
each set of rounds~$\rounds_\ell$.

\paragraph{Home-away} (HA)
Each team has a home venue, implying that to specify a
schedule it is no longer sufficient to specify the matches in each round;
instead, one also has to specify, for each match, which teams plays home,
and which team plays away.
We denote this by redefining a match between teams $i, j \in \teams$
($i \neq j$) where $i$ is the home-playing team, by an ordered pair $(i, j)$ (in contrast to an unordered pair $\{i, j\}$).

Let~$k$ be a positive integer.
For $n$ teams and a $k$-round robin setting,
denote $\teams \define \{1, \dots, n\}$
and $\rounds \define \{1, \dots, k(n-1) \}$.
Let $\matches \define \{ (i, j) : i, j \in \teams, i \neq j \}$ be the set
of ordered matches. The assignment of match $(i, j) \in \matches$ to round
$r \in \rounds$ comes at a cost $c_{(i, j), r}$,
and in contrast to the SRR case, $c_{(i, j), r}$ and $c_{(j, i), r}$
can be different.
We proceed by describing the phased $k$-round robin problem with home-away patterns ($k$RR-PH-HA):

\begin{problem}[$k$RR-PH-HA]
\label{prob:krrha}
Given an even number~$n \geq 4$ of teams with corresponding matches~$\matches$,
  a set of~$\frac12kn(n-1)$ rounds~$\rounds$ ($k \geq 1$), as well as an integral cost~$c_{\match,r}$ for
  every~$\match\in\matches$ and round~$r \in \rounds$, the $k$RR-PH-HA problem is to find an assignment~$\mathcal{A}
  \subseteq \matches \times \rounds$ of matches to rounds such that (i) every
  team plays a single match per round, (ii) each pair of teams meets once in part~$R_{\ell}, \ell \in \{1, \ldots, k\}$, and (iii) each ordered match is played $\lfloor \frac{k}{2} \rfloor$ or~$\lceil \frac{k}{2} \rceil$ times so that each pair of teams meets in total $k$ times, while total cost $\sum_{(\match,r) \in \mathcal{A}} c_{\match,r}$ is minimized.
\end{problem}

We will show how the three formulations of Sections~\ref{sec:traditionalformulation}--\ref{sec:permutationformulation} can be adapted
to deal with these properties.

\paragraph{Extending the traditional formulation for $k$-RR tournaments}

For reasons of convenience, we assume $k$ is even; this implies that for each pair of distinct teams $i,j \in T$ match~$(i,j)$ and match~$(j,i)$ each need to occur $\frac{k}{2}$ times in any feasible schedule.
\begin{subequations}
  \makeatletter
  \def\@currentlabel{$k$T}
  \makeatother
  \renewcommand{\theequation}{$k$T\arabic{equation}}%
  \label{trakRR}
  \begin{align}
    \min \sum_{(i,j) \in \matches} \sum_{r \in \rounds} c_{(i,j),r}x_{(i,j),r} &&&\label{trakRR:obj}\\
    \sum_{r \in \rounds} x_{(i,j),r} &= \frac{k}{2}, && (i,j) \in \matches,\label{trakRR:eachmatch}\\
    \sum_{r \in \rounds_\ell} (x_{(i,j),r} + x_{(j,i),r}) &= 1, && i,j \in \teams,\; i\neq j,\; \ell \in \{0,\dots,k-1\}, \label{trakRR:eachpart}\\
    \sum_{j \in \teams\setminus\{i\}} (x_{(i,j),r} + x_{(j,i),r}) &= 1, && i \in \teams,\; r \in \rounds,\label{trakRR:eachteameachround}\\
    x_{(i,j),r} &\in \{0,1\}, && (i,j) \in \matches,\; r \in \rounds. \label{trakRR:int}
  \end{align}
\end{subequations}

Constraints \eqref{trakRR:eachmatch} ensure that each match is played $\frac{k}{2}$ times, Constraints \eqref{trakRR:eachpart} express that each pair of teams has to meet once in each part (the ``phased'' property), and Constraints \eqref{trakRR:eachteameachround} prescribe that each team plays a single match in each round.

\paragraph{Extending the matching formulation for $k$RR tournaments}

We now assume that $K_n=(T,A)$ is a {\em directed} multi-graph, where each
arc $(i,j)$ with~$i,j \in T$, $i \neq j$ is present $\frac{k}{2}$ times; a
(directed) matching $\matching$ is now defined as a set of $\frac{n}{2}$
arcs incident to each node once, and $\matchings$ now stands for the set of
all (directed) matchings.

\begin{subequations}
  \makeatletter
  \def\@currentlabel{$k$M}
  \makeatother
  \renewcommand{\theequation}{$k$M\arabic{equation}}%
  \label{matkRR}
  \begin{align}
    \min \sum_{\matching\in\matchings} \sum_{r\in\rounds} d_{\matching,r} y_{\matching,r} &&&\label{matkRR:obj}\\
    \sum_{\matching\in\matchings} y_{\matching,r} &=1, && r\in\rounds,\label{matkRR:eachround}\\
    \sum_{r\in\rounds}\sum_{\substack{\matching\in\matchings\colon\\(i,j)\in\matching}} y_{\matching,r} &= \frac{k}{2}, && (i,j) \in \matchings,\label{matkRR:eachmatch}\\
    \sum_{r\in\rounds_{\ell}} \sum_{\substack{\matching\in\matchings\colon\\(i,j)\in\matching \text{ or } (j,i)\in\matching}} y_{\matching,r} &=1, && i,j \in T,\; i \neq j,\; \ell \in \{0,\dots,k-1\},\label{matkRR:eachpart}\\
    y_{\matching,r} &\in \{0,1\}, && \matching\in\matchings,\; r \in\rounds.
  \end{align}
\end{subequations}
Constraints \eqref{matkRR:eachround} ensure that a (directed) perfect matching is selected in each round, Constraints~\eqref{matkRR:eachmatch} say that each match occurs $\frac{k}{2}$ times, and Constraints \eqref{matkRR:eachpart} model the fact that each pair of teams has to meet once in each part.

\paragraph{Extending the permutation formulation for $k$RR tournaments}

We have to redefine $\Pi^{-i}$: each entry needs to specify home or away,
and of course, the fact that each team meets all other teams in rounds
$(\ell-1)(n-1)+1, \ldots, \ell (n-1)$, for each $\ell \in \{1,\dots,k\}$ needs
to be taken into account. Other than that Formulation \eqref{per} remains
unchanged.

Without giving formal proofs, we claim that the linear relaxations of the extension of the matching formulation, as well as the linear relaxation of the extension of the permutation formulation can be solved in polynomial time. We also claim that the extension of the matching formulation is stronger than the other two formulations---the adaptations to the proofs of Lemmata~\ref{lem:pricematch} and \ref{lem:pricepermutation} and Theorem~\ref{th:overallstrength} are straightforward.

\section{Computational results}
\label{sec:computationalresults}
In this section, we report the outcomes of our computational experiments.
Section~\ref{sec:instances+solvers} describes the test set that we have
used in our experiments.
Afterwards, we investigate the quality of the LP relaxations of the
different models and compare their corresponding values in
Section~\ref{sec:computationalcomparison}.
Finally, we discuss our experience with solving instances of the SRR problem using the matching
formulation, i.e., not only solving the LP relaxation but also the corresponding
integer program.
To this end, we have implemented a branch-and-price algorithm whose details
are described in Section~\ref{sec:branchandprice}.

\subsection{Test set}
\label{sec:instances+solvers}

We have generated~\num{1000} instances\footnote{The instances as well as
the implementation of our algorithms are publicly available at
\url{https://github.com/JasperNL/round-robin}; all experiments have been conducted using
the code with githash \texttt{1657b4d7}.
} of the SRR problem to evaluate the quality of the
LP relaxations of the different models.
Our test set comprises of instances of different sizes and thus different levels
of difficulty, which are parameterized by a tuple~$(n, \rho)$ and have
cost coefficients attaining values~$0$ or~$1$.
Parameter~$n$ encodes the number of teams and has range~$n \in \{6, 12, 18,
24\}$;
parameter~$\rho$ controls the number of~1-entries in the objective.
More precisely, we pick a set of match-round
pairs~$\matches \times \rounds$ of size~$\lfloor \rho \cdot \card{\matches
  \times \rounds} \rfloor$ uniformly at random, denoted by $S \subseteq \matches \times
\rounds$, where $\rho \in \{0.5, 0.6, 0.7, 0.8, 0.9\}$.
The generated instance consists of $n$ teams
and has cost coefficients~$c_{m,r} = 1$ if~$(m, r) \in S$
and~$c_{m,r} = 0$ otherwise. For each combination of $n \in \{6, 12, 18,
24\}$ and $\rho \in \{0.5, 0.6, 0.7, 0.8, 0.9\}$ we have generated 50
instances.

\subsection{A computational comparison of the linear relaxations}
\label{sec:computationalcomparison}
In this section, we provide a computational comparison between the
LP relaxation values of the traditional
formulation~\eqref{tra} (and thus by 
Lemma~\ref{lem:equitraandper} also of the permutation formulation)
and LP relaxation values of the matching
formulation~\eqref{mat},
and compare these to the actual optimal (or best found) integral solutions.
Before we discuss our numerical results, we provide details about our
implementation as well as on how we find optimal integral solutions first.

\paragraph{Implementation details}
To find the LP relaxation values, 
we implemented both formulations 
in Python~3 using the \texttt{PySCIPOpt} 4.1.0
package~\cite{MaherMiltenbergerPedrosoRehfeldtSchwarzSerrano2016} for
\texttt{SCIP}~8.0.0~\cite{BestuzhevaEtal2021OO}, 
with \texttt{CPLEX} 20.1.0.0 as LP solver.
The traditional formulation is implemented as a compact model.
For the matching formulation, we use a column generation procedure that
receives a subset of all variables, solves the corresponding LP
relaxation restricted to these variables, and adds further variables
until it can prove that an optimal LP solution has been found.
To identify whether new variables need to be added, we solve the so-called
pricing problem, which corresponds to separating a corresponding solution
of the dual problem.
The separation problem can be solved by finding a maximum weight perfect
matching as detailed in the proof of Lemma~\ref{lem:pricematch}.
We start with the empty set of variables, which 
means that the primal problem is initially infeasible. 
Analogously to the proof of Lemma~\ref{lem:pricematch}, we resolve 
infeasibility by adding variables to the problem that are associated with 
a dual constraint that violate a dual unbounded ray of this infeasible 
problem.
The column generation procedure has been embedded in a so-called 
\texttt{pricer} plug-in of \texttt{SCIP}, which adds newly generated 
variables to the matching formulation.
The maximal weight perfect matchings are computed using 
\texttt{NetworkX}~2.5.1, which provides an implementation of Edmonds' 
blossom algorithm.

\paragraph{Finding optimal integer solutions}
To obtain the optimal integer solution value of as many instances as
possible, we have used two different solvers to solve the integer program
of Model~\eqref{tra}.
On the one hand, we have used \texttt{SCIP} as described in the above
setup.
On the other hand, we have modeled~\eqref{tra} using \texttt{Gurobi}~9.1.2
via its Python~3 interface.
For each instance and solver, we have imposed a time limit of~\SI{48}{\hour} to
find an optimal integer solution.
Using \texttt{SCIP}, we managed to solve 852 of the 1000 instances to 
optimality. With \texttt{Gurobi}, we were able to solve 866 of the 1000 
instances to optimality.
There were 45 instances where \texttt{SCIP} found a better primal 
objective value, and 79 instances where \texttt{Gurobi} found a better 
primal objective value.
All experiments have been run on a compute cluster with 
identical machines, using one (resp. two) thread(s) 
on Xeon Platinum 8260 processors, with \SI{10.7}{\giga\byte} 
(resp. \SI{21.4}{\giga\byte}) memory, 
respectively for \texttt{SCIP} and \texttt{Gurobi}. 

\begin{table}[!tbp]
\caption{Comparison of the LP relaxation values of the traditional and
  matching formulation.}
\label{tab:comp}
\footnotesize
\centering

\begin{tabular}{
@{}
r@{\hspace{4pt}}
r
r@{\ }r@{\ }r
c@{\ }c
c@{\hspace{4pt}}
r@{\ }r@{\ }r
c@{\ }c
r@{\hspace{4pt}}r
@{}
}
\toprule

&&\multicolumn{5}{c}{all instances}
&\multicolumn{8}{c}{restricted to instances with $\lprelaxation{\rm tra} < \ipsolution$} \\
\cmidrule(l{4pt}r{4pt}){3-7}
\cmidrule(l{4pt}r{4pt}){8-15}

&&\multicolumn{3}{c}{average value}
&\multicolumn{2}{c}{solved}
&&\multicolumn{3}{c}{average value}
&\multicolumn{2}{c}{solved}
&\multicolumn{2}{c}{gap closed} \\
\cmidrule(l{4pt}r{4pt}){3-5}
\cmidrule(l{4pt}r{4pt}){6-7}
\cmidrule(l{4pt}r{4pt}){9-11}
\cmidrule(l{4pt}r{4pt}){12-13}
\cmidrule(l{4pt}r{4pt}){14-15}

$n$ & $\rho$ & $\lprelaxation{\rm tra}$ & $\lprelaxation{\rm mat}$ & 
$\ipsolution$ & O & T & 
\# & $\lprelaxation{\rm tra}$ & $\lprelaxation{\rm mat}$ & 
$\ipsolution$ & O & T & 
average & maximal \\
\midrule
  6 & 0.5 &  2.227 &  2.297 &  2.380 & 50 &  0 & 
             14 &  2.452 &  2.702 &  3.000 & 14 &  0 &    43.45\% &   
             100.00\% \\
  6 & 0.6 &  3.802 &  3.865 &  3.920 & 50 &  0 & 
             12 &  3.924 &  4.188 &  4.417 & 12 &  0 &    52.08\% &   
             100.00\% \\
  6 & 0.7 &  5.430 &  5.510 &  5.540 & 50 &  0 & 
              9 &  5.611 &  6.056 &  6.222 &  9 &  0 &    66.67\% &   
              100.00\% \\
  6 & 0.8 &  7.620 &  7.635 &  7.660 & 50 &  0 & 
              4 &  7.250 &  7.438 &  7.750 &  4 &  0 &    37.50\% &   
              100.00\% \\
  6 & 0.9 & 10.003 & 10.040 & 10.060 & 50 &  0 & 
              6 & 10.361 & 10.667 & 10.833 &  6 &  0 &    66.67\% &   
              100.00\% \\
\midrule
 12 & 0.5 &  0.080 &  0.080 &  0.080 & 50 &  0 & 
              0 & --     & --     & --     &  0 &  0 & --         & 
              --         \\
 12 & 0.6 &  2.018 &  2.213 &  3.480 & 50 &  0 & 
             49 &  2.019 &  2.217 &  3.510 & 49 &  0 &    15.29\% &   
             100.00\% \\
 12 & 0.7 &  8.022 &  8.342 &  9.500 & 50 &  0 & 
             50 &  8.022 &  8.342 &  9.500 & 50 &  0 &    22.27\% &    
             70.77\% \\
 12 & 0.8 & 17.184 & 17.474 & 18.340 & 50 &  0 & 
             50 & 17.184 & 17.474 & 18.340 & 50 &  0 &    29.89\% &   
             100.00\% \\
 12 & 0.9 & 31.459 & 31.654 & 31.840 & 50 &  0 & 
             31 & 31.096 & 31.410 & 31.710 & 31 &  0 &    56.87\% &   
             100.00\% \\
\midrule
 18 & 0.5 &  0.000 &  0.000 &  0.000 & 50 &  0 & 
              0 & --     & --     & --     &  0 &  0 & --         & 
              --         \\
 18 & 0.6 &  0.060 &  0.060 &  0.060 & 50 &  0 & 
              0 & --     & --     & --     &  0 &  0 & --         & 
              --         \\
 18 & 0.7 &  2.045 &  2.292 &  5.600 & 50 &  0 & 
             50 &  2.045 &  2.292 &  5.600 & 50 &  0 &     6.68\% &    
             15.66\% \\
 18 & 0.8 & 19.831 & 20.330 & 23.900 & 50 &  0 & 
             50 & 19.831 & 20.330 & 23.900 & 50 &  0 &    12.37\% &    
             28.19\% \\
 18 & 0.9 & 52.700 & 53.066 & 54.500 & 50 &  0 & 
             49 & 52.673 & 53.047 & 54.510 & 49 &  0 &    21.55\% &   
             100.00\% \\
\midrule
 24 & 0.5 &  0.000 &  0.000 &  0.000 & 50 &  0 & 
              0 & --     & --     & --     &  0 &  0 & --         & 
              --         \\
 24 & 0.6 &  0.000 &  0.000 &  0.000 & 50 &  0 & 
              0 & --     & --     & --     &  0 &  0 & --         & 
              --         \\
 24 & 0.7 &  0.200 &  0.200 &  4.340 &  0 & 50 & 
             49 &  0.163 &  0.163 &  4.408 &  0 & 49 &     0.00\% &     
             0.00\% \\
 24 & 0.8 & 12.352 & 12.893 & 24.180 &  0 & 50 & 
             50 & 12.352 & 12.893 & 24.180 &  0 & 50 &     4.57\% &     
             7.91\% \\
 24 & 0.9 & 69.327 & 69.922 & 74.860 & 29 & 21 & 
             50 & 69.327 & 69.922 & 74.860 & 29 & 21 &    10.69\% &    
             21.68\% \\
\bottomrule
\end{tabular}

\end{table}

\paragraph{Numerical results}
Table~\ref{tab:comp} shows the aggregated computational results of our experiments.
For each number of teams $n$ and ratio $\rho$,
we provide the average of the objective values of the relaxation of the traditional formulation (column ``$\lprelaxation{\rm tra}$''), 
the average of the objective values of the relaxation of the matching formulation (column ``$\lprelaxation{\rm mat}$''), and the average optimum value (column ``$\ipsolution$'').
Notice that for $n=24$ we have not been able to solve all instances to optimality; in this case, we use the value of the best
known solution instead of the (unknown) optimum for that instance in
the~$\ipsolution$ column. Recall that each value is an average over 50
instances.
The number of optimally solved instances (resp.\ instances not terminating
within the time limit) are shown in column ``O'' (resp.\ ``T'').

To be able to assess the strength of the matching formulation compared to the traditional formulation, we focus, in the right side of the table, on those instances for which $\lprelaxation{\rm tra} < \ipsolution$; their number (out of 50) is given in the column labeled ``\#''. From this column, we see that the fraction~$\rho$ that leads to instances with a gap between $\lprelaxation{\rm tra}$ and $\ipsolution$ slowly increases with $n$. Indeed, for $n=6$, most instances do not have a gap, for $n=12$, almost all instances with $\rho \in \{0.6, 0.7, 0.8\}$ have a gap, and for $n=18$, almost all instances with $\rho \in \{0.7, 0.8, 0.9\}$ have a gap.

We use the notion of the \emph{relative gap} that is closed by the matching formulation relative to
the traditional formulation, given by
\[
{\rm rgap}(I) \define \frac{
\lprelaxation{\rm mat}(I) - \lprelaxation{\rm tra}(I)}{
\ipsolution(I) - \lprelaxation{\rm tra}(I)}\ %
{\rm for~an~instance~}I~{\rm of~SRR} \text{ with } \ipsolution(I) - \lprelaxation{\rm tra}(I) > 0.
\]
A value of zero for ${\rm rgap}(I)$ implies that the relaxation values of the traditional
formulation and the matching formulation are equal, while a value of one
(i.e., 100\%) implies that the relaxation of the matching formulation is
equal to the true objective of the optimal integral
solution. The column ``average'' gives the average {\rm rgap}, whereas
column ``maximal'' shows the maximum relative closed gap for an instance of
this sub test set.

For $n=6$, there are few instances with a gap. However, for those instances for which there is a gap, it is clear that a sizable part of that gap is closed by the relaxation of the matching formulation.
For larger values of~$n$, many instances have a gap. We observe that a significant percentage of the gap is closed by the relaxation of the matching formulation.
If~$n$ is getting larger, however, both the value of the average gap closed as well as the value of the maximal gap closed decrease.
We conclude that for small values of~$n$, and thus for many realistic
applications, the matching formulation provides a much better relaxation
value than the traditional formulation.

\subsection{A branch-and-price algorithm}
\label{sec:branchandprice}

Since the matching formulation can dominate the traditional formulation, a
natural question is whether the stronger formulation also allows to solve
the SRR problem faster than the traditional formulation.
For this reason, we have implemented a branch-and-price algorithm (in the
computational setup as described above) to compute optimal integral
solutions of the matching formulation.
That is, we use a branch-and-bound algorithm to solve the matching
formulation, where each LP relaxation is solved using a column generation
procedure.

\paragraph{Implementation details}
In classical branch-and-bound algorithms, the most common way to implement
the branching scheme is to select a variable~$x_i$ whose value~$x^\star_i$
in the current LP solution is non-integral and to generate two
subproblems by additionally enforcing either~$x_i \leq \lfloor x^\star_i
\rfloor$ or~$x_i \geq \lceil x^\star \rceil$.
In principle, this strategy is also feasible for the matching formulation,
where the subproblems correspond to forbidding a
schedule~$\matching\in\matchings$ for a round~$r \in \rounds$ or fixing the
schedule in round~$r$ to be~$\matching$.
This branching scheme, however, leads to a very unbalanced branch-and-bound
tree as the former subproblem only rules out a very specific schedule,
while the latter one fixes the matches of an entire round.
Another difficulty of the classical scheme is that it might affect the
structure of the pricing problem in the newly generated subproblems.
Ideally, the pricing problem should not change such that the same algorithm
can be used for adding new variables to the problem.
We will address both issues next.

To obtain a more balanced branch-and-bound tree, we have implemented a
custom branching rule following the Ryan-Foster branching
scheme~\cite{RyanFoster1981}:
Our scheme selects a match~$\{i, j\} \in \matches$ at a round~\mbox{$r \in
  \rounds$} and creates two children.
In the left child, we forbid that $\{i, j\}$ is played in round~$r$,
and in the right child, we enforce that $\{i, j\}$ is played in round $r$.
Note that for all matchings~$\matching \in \matchings$
this branching decision fixes all variables $y_{\matching,r}$ to 
zero if $\{i, j\} \in \matching$ for the left child,
and $\{i, j\} \notin \matching$ for the right child.

Using this branching strategy, the structure of the pricing problem at each
subproblem remains a matching problem.
At the root node of the branch-and-bound tree, we need to solve a maximum
weight perfect matching problem in a weighted version of~$K_n$ as described
above.
At other nodes of the branch-and-bound tree, we have added branching
decisions that enforce that two teams~$i$ and~$j$ either do meet or do not meet
in a round~$r \in \rounds$.
These decisions can easily be incorporated by deleting edges from~$K_n$.
When generating variables for round~$r$, we remove edge~$\{i,j\}$
from~$K_n$ if~$i$ and~$j$ shall not meet in this round; if the
match~$\{i,j\}$ shall take place, then we remove all edges incident
with~$i$ and~$j$ except for~$\{i,j\}$.
Consequently, our branching strategy allows to solve the LP relaxations of
all subproblems in polynomial time.

Since our Python implementation of the traditional and matching formulation
took too much time to be used in a branching scheme, we decided to implement
our branch-and-price algorithm as a plug-in using the C-API of \texttt{SCIP}.
The pricer plug-in is analogous, and maximal weight perfect matchings are 
now computed using the \texttt{LEMON}~1.3.1 graph library.
To ensure that the branching decisions are taken into account, we also
implemented a constraint that fixes $y_{\matching, r}$ to zero if the
matching~$\matching$ violates the branching decisions for round~$r$, and
added a plug-in that implements the branching decisions. 

The branching rule sketched above admits some degrees of freedom in
selecting the match~$\{i,j\}$ and round~$r$.
In our implementation, we decided to mimic two well-known branching rules:
most infeasible branching and strong branching on a selection of 
variables, see Achterberg \etal~\cite{achterberg2005branchingrulesrevised} for an overview on branching rules.
Most infeasible branching branches on a binary variable with
fractional value in an LP solution that is closest to~0.5,
and strong branching branches on the variable that yields the 
largest dual bound improvement based on some metric. 
Since strong branching requires significant computational effort, it is common to make 
a limited branching candidate selection and apply strong branching on 
those.

\begin{algorithm}[!tbp]
\SetKwInOut{Input}{input}
\SetKwInOut{Output}{output}
\caption{Determining the branching candidate for an LP node.}
\label{alg:branchingrule}
\Input{An LP solution $y^\star_{m, r}$ in a branch-and-bound tree node at 
depth $d$ with objective ${\rm obj}$.}
\Output{The branching decision
(match $m \in \matches$ on round $r \in \rounds$),
or detected integrality.}

\tcp{fractional assignment of match $m$ to round $r$}

compute 
${\rm assign}_{m, r} \gets \sum_{M \in \matchings\colon \match\in\matching} y^\star_{M,r}$ 
for $m \in \matches$ and $r \in \rounds$%
\;

\BlankLine

\If{${\rm assign}_{m, r}$ is $0$ or $1$ 
for all $m \in \matches$ and $r \in \rounds$}%
{
	\Return integral solution found\;
}

\BlankLine

\tcp{fractional part of ${\rm assign}_{m, r}$}
compute 
${\rm frac}_{m, r} \gets 
\min\{ {\rm assign}_{m, r}, 1.0 - {\rm assign}_{m, r} \}$
for $m \in \matches$ and $r \in \rounds$%
\;

\BlankLine

\tcp{score for every match-round pair}
compute 
$
{\rm score}_{m, r} 
\gets 
{\rm frac}_{m, r}
\cdot 
(1.0 + | c_{m, r} |)
\cdot 
({\rm assign}_{m, r})^2
$
for $m \in \matches$ and $r \in \rounds$%
\label{alg:branchingrule:score}
\;

\BlankLine

\tcp{strong branching candidate selection}
$
{\rm number\_of\_candidates} \gets 
\max\{
1,
\lfloor
0.1 \cdot \card{\matches \times \rounds} \cdot 0.65^{d}
\rfloor
\}
$\;

\If{${\rm number\_of\_candidates} > 1$}{
	pick ${\rm number\_of\_candidates}$ 
	candidates $(m, r) \in \matches \times \rounds$
	with highest ${\rm score}_{m, r}$
	as strong branching candidates\;
	\ForEach{strong branching candidate $(m, r)$}{
		\If{${\rm score}_{m, r} = 0.0$}{
			\tcp{then ${\rm assign}_{m, r}$ is 0 or 1, skip this candidate}
			\Continue\ (i.e., skip this candidate)\;
		}
		apply strong branching on $(m, r)$,
		with objectives
		${\rm obj}_{\rm forbid}$, ${\rm obj}_{\rm enforce}$
		in the two children\;
		compute ${\rm score}_{m, r}^\star \gets 
		({\rm obj}_{\rm forbid} - {\rm obj} + 1.0) \cdot 
		({\rm obj}_{\rm enforce} - {\rm obj} + 1.0)$\;
		\label{alg:branchingrule:scorestar}
	}
	\Return branch on strong branching candidate $(m, r)$ 
	with maximal~${\rm score}_{m, r}^\star$\;
}
\Else{
	\tcp{do not apply strong branching deep in the branch-and-bound tree}
	\Return branch on $(m, r) \in \matches \times \rounds$
	with maximal ${\rm score}_{m, r}$\;
}

\end{algorithm}

The pseudocode for our branching rule 
is given in Algorithm~\ref{alg:branchingrule}.
We start by computing the fractional match-on-round assignment values
induced by the $y$-variables.
Then, we make a selection of 
potentially good branching candidates~$(m, r) \in \matches \times \rounds$,
that is based on the ${\rm score}_{m, r}$ metric
shown in Line~\ref{alg:branchingrule:score}:
this score prioritizes match-on-round assignments for which 
\begin{enumerate*}[label=(\roman*), ref=(\roman*)]
\item\label{score1} the assignment value is close to $0.5$,
\item\label{score2} the cost coefficients are large,
and
\item\label{score3} the assignment values are relatively high.
\end{enumerate*}
Using this score, we hope to resolve fractionality soon (by~\ref{score1}).
By~\ref{score2}, we want to enforce a significant change of the objective
value in the child that forbids match~$\match$, whereas the child enforcing
that~$\match$ is played selects a match that is most likely played due
to~\ref{score3}.
Experiments show that full strong branching leads to a smaller number of 
nodes to solve the problem, but this turns out to be very costly computationally.
Therefore, we only apply strong branching for branch-and-bound tree nodes 
close to the root, and only evaluate a subset of branching candidates 
that have the highest ${\rm score}_{m, r}$-metric.
This is our candidate pre-selection. 
The higher the depth of the considered branch-and-bound tree node, the 
smaller the number of candidates considered.
Of those branching candidates, we pick the candidate that 
maximizes~${\rm score}_{m, r}^{\star}$ as defined in 
Line~\ref{alg:branchingrule:scorestar}.
The goal of this score is to choose the candidate where the objective 
values of the hypothetical children are different from the current node's 
objective. By considering the product of this difference, we prioritize
if the objective of both hypothetical children have some difference with the 
current node objective.
If the number of candidates is only one, no strong branching is applied 
and that candidate match-on-round assignment is chosen for branching.

All experiments have been run on a Linux cluster with Intel Xeon E5
\SI{3.5}{\GHz} quad core processors and~\SI{32}{\giga\byte} memory.
The code was executed using a single thread and the time limit for all
computations was~\SI{2}{\hour} per instance.

\paragraph{Numerical results}
Table~\ref{tab:comp:branch} summarizes our results for our
instances for~$n \in \{6,12,18\}$.
We distinguish the instances by their parameters~$n$ and~$\rho$, and we
report on the number of instances that could be solved (resp.\ could not be
solved) within the time limit in column ``O'' (resp. ``T'').
Moreover, we report on the the minimum, mean, and maximum running time per
parameterization.
The mean of all running times~$t_i$ is reported in shifted geometric mean~$\prod_{i =
  1}^{50} (t_i + s)^{\frac{1}{50}} - s$ using a shift of~\SI{10}{\second}
to reduce the impact of instances with very small running times.

\begin{table}[t]
\caption{Computational results for the branch-and-price algorithm for Model~\eqref{mat}.}
\label{tab:comp:branch}
\footnotesize
\centering

\begin{tabular}{@{}rrcccrrr@{}}
\toprule
&&&\multicolumn{2}{c}{Solved}&
\multicolumn{3}{c}{Solving time (s)}\\
\cmidrule(l{4pt}r{4pt}){4-5}
\cmidrule(l{4pt}r{4pt}){6-8}
$n$ & $\rho$ & \# & O & T & min & mean & max \\
\midrule
  6&0.5&50&50& 0&    0.00&    0.00&    0.01\\
  6&0.6&50&50& 0&    0.00&    0.00&    0.01\\
  6&0.7&50&50& 0&    0.00&    0.00&    0.01\\
  6&0.8&50&50& 0&    0.00&    0.00&    0.01\\
  6&0.9&50&50& 0&    0.00&    0.00&    0.01\\
\midrule
 12&0.5&50&50& 0&    1.25&    2.38&    5.73\\
 12&0.6&50&50& 0&    0.12&    4.09&    8.28\\
 12&0.7&50&50& 0&    0.21&    3.33&    7.38\\
 12&0.8&50&50& 0&    0.14&    2.05&    7.63\\
 12&0.9&50&50& 0&    0.11&    0.54&    2.21\\
\midrule
 18&0.5&50&50& 0&   54.61&  106.09&  210.77\\
 18&0.6&50&48& 2&  103.41&  866.21& 7200.00\\
 18&0.7&50& 3&47&  930.53& 6854.47& 7200.09\\
 18&0.8&50&22&28&  312.65& 4084.21& 7200.06\\
 18&0.9&50&50& 0&    6.74&  332.98& 3990.61\\
\bottomrule
\end{tabular}
\end{table}

We observe that instances with~6 and~12 teams can be solved very
efficiently within fractions of seconds in the former and within seconds in
the latter case.
Instances with~18 teams are more challenging, in particular, if the
ratio~$\rho \in \{0.7, 0.8\}$.
In this case, only~3 and~22 instances could be solved, respectively, but
note that not all instances are equally difficult.
For instance, for~$n = 18$ and~$\rho = 0.8$, there exists an instance that
can be solved within roughly five minutes, whereas the mean running time is
more than an hour.
To fully benefit from the strong LP relaxation of the matching formulation,
it might be the case that additional algorithmic enhancements can further improve the
performance of the branch-and-price algorithm.

\section{Conclusion}
\label{sec:conclusion}

The use of integer programming for finding schedules of round robin tournaments is widespread. We have introduced and analyzed two new formulations for this problem, one of which (the matching formulation) is stronger than the other formulations. We have proposed a class of valid inequalities for the matching formulation, which may be of use when developing cutting-plane based techniques for this problem. By randomly generating instances, we studied the strength of the formulations, and we implemented a branch-and-price algorithm based on the matching formulation to see its efficiency.
Although this algorithm is able to solve small-scale instances rather
efficiently, solving large instances of the SRR efficiently remains a challenge.

Possible directions of future research are thus to further strengthen our
integer programming formulations/techniques.
On the one hand, once can investigate additional cutting planes to strengthen both the traditional
and matching formulation.
For the matching formulation, cutting planes will in particular affect the pricing
problem and thus might change its structure.
Thus, the trade-off between the strength of cutting planes and the
difficulty of solving the pricing problem that needs to be investigated.
On the other hand, one can enhance our branch-and-price algorithm in
several directions, e.g., the development of more sophisticated branching
rules or heuristics for producing good schedules.

\paragraph{Acknowledgment}
The fourth author's research is supported by the Dutch Research Council (NWO) through Gravitation grant NETWORKS-024.002.003.

\bibliographystyle{abbrv}

\end{document}